\documentclass[leqno,11pt,a4paper]{amsart}
\usepackage[full]{textcomp}
\usepackage{newpxtext}
\usepackage{cabin} 
\usepackage[varqu,varl]{inconsolata} 
\usepackage[bigdelims,vvarbb]{newpxmath} 
\usepackage[cal=cm,bb=esstix,bbscaled=1.05]{mathalfa} 
\usepackage[margin=2.4cm]{geometry}
\usepackage{savesym}

\usepackage{mathabx}
\usepackage[usenames,dvipsnames]{color}
\usepackage{hyperref}
\hypersetup{
 colorlinks,
 linkcolor={red!50!black},
 citecolor={blue!50!black},
 urlcolor={blue!80!black}
}
\usepackage{tikz}
\usepackage{graphicx}
\usepackage[all,cmtip]{xy}
\usepackage{mleftright}
\usepackage{booktabs}
\usepackage{cite}
\usepackage{enumitem}
\usepackage{mathdots}

\setlist[enumerate]{labelsep=*, leftmargin=1.5pc}
\setlist[enumerate]{label=\normalfont(\roman*), ref=\roman*}
\newtheorem{thm}{Theorem}[section]
\newtheorem{thm*}[]{Theorem}
\newtheorem{lemma}[thm]{Lemma}
\newtheorem{cor}[thm]{Corollary}
\newtheorem{prop}[thm]{Proposition}
\newtheorem{prop*}[]{Proposition}

\newtheorem{conjecture}[thm]{Conjecture}
\newtheorem{conjecture*}[]{Conjecture}
\theoremstyle{definition}
\newtheorem{example}[thm]{Example}

\newtheorem{definition}[thm]{Definition}
\newtheorem{definition*}[]{Definition}

\numberwithin{equation}{section}
\newcommand{\Z}{\mathbb{Z}}
\newcommand{\Q}{\mathbb{Q}}
\newcommand{\R}{\mathbb{R}}
\newcommand{\C}{\mathbb{C}}

\newcommand{\bF}{\mathbb{F}}
\newcommand{\bK}{\mathbb{K}}

\newcommand{\mO}{\mathcal{O}}
\newcommand{\mA}{\mathcal{A}}

\newcommand{\eps}{\varepsilon}

\ProvideTextCommand{\DJ}{OT1}{\raisebox{0.25ex}{-}\kern-0.4em D}

\newcommand{\pr}{\mathbb{P}}

\newcommand{\on}{\operatorname}
\newcommand{\ol}{\overline}
\newcommand{\wt}{\widetilde}

\newcommand{\mfk}{\mathfrak}

\newcommand{\wh}{\widehat}
\newcommand{\mat}{\left(\begin{array}}
\newcommand{\tam}{\end{array}\right)}

\newcommand{\ralg}{\on{Ru}^{\textnormal{alg}}}
\newcommand{\ealg}{e^{\textnormal{alg}}}
\newcommand{\calg}{c^{\textnormal{alg}}}
\newcommand{\cech}{c^{\textnormal{ECH}}}
\newcommand{\casy}{c^{\textnormal{asy}}}
\newcommand{\lf}{\lfloor}
\newcommand{\rf}{\rfloor}
\newcommand{\lc}{\left\lceil}
\newcommand{\rc}{\right\rceil}
\newcommand{\ip}{\text{\textnormal{IP}}}

\newcommand{\nefqc}{\on{nef}^\text{\textnormal{qc}}}
\newcommand{\Nefqc}{\on{Nef}^\text{\textnormal{qc}}}

\newenvironment{dedication}
  {
   \vspace*{-15pt}
   \itshape             
   \center          
  }
  {
   \par
   \vspace{1pt}
  }
\graphicspath{{images/}}
\begin{document}
\author[B.~Wormleighton]{B.~Wormleighton}
\address{Department of Mathematics\\University of California at Berkeley\\Berkeley, CA\\94720\\USA}
\email{b.wormleighton@berkeley.edu}
\title[]{Algebraic capacities}
\maketitle

\begin{dedication}
Black lives matter.
\end{dedication}

\begin{abstract} We study invariants coming from certain optimisation problems for nef divisors on surfaces. These optimisation problems arise in work of the author and collaborators tying obstructions to embeddings between symplectic $4$-manifolds to questions of positivity for (possibly singular) algebraic surfaces. We develop the general framework for these invariants and prove foundational results on their structure and asymptotics. We describe the connections these invariants have to Embedded Contact Homology (ECH) and the Ruelle invariant in symplectic geometry, and to min-max widths in the study of minimal hypersurfaces. We use the first of these connections to obtain optimal bounds for the sub-leading asymptotics of ECH capacities for many toric domains.
\end{abstract}

\section{Introduction}

Algebraic capacities are invariants of a polarised algebraic surface that arose from work of the author and collaborators \cite{bwo,bwt,cw} to study symplectic embedding obstructions using techniques from algebraic geometry. In this paper we develop the general framework to incorporate these invariants into the wider setting of algebraic positivity \cite{laz1,laz2}. Algebraic capacities also connect to other areas of mathematics such as minimal hypersurfaces, convex optimisation, and symplectic geometry as already mentioned. We utilise some landmark ideas and results in these fields to inspire structural results for algebraic capacities in purely algebro-geometric terms. We will especially press into the connection between algebraic capacities and Embedded Contact Homology (ECH) to produce refined information about symplectic embedding obstructions and their asymptotics using algebro-geometric methods.

\subsection{Algebraic capacities as positivity invariants}

Throughout the paper a
$$\text{\emph{weakly polarised / pseudo-polarised / polarised surface}}$$
means a pair $(Y,A)$ with $Y$ a normal projective algebraic surface and $A$ an $\R$-divisor on $Y$ that is
$$\text{big / big and nef / ample}$$
and $\Q$-Cartier. Adjectives applied to $(Y,A)$ are understood to apply to whichever of $Y$ or $A$ it makes sense for them to apply to. For instance, a polarised smooth surface $(Y,A)$ is a pair consisting of a smooth projective algebraic surface $Y$ and an ample $\R$-divisor on $Y$.

\begin{definition*} Given a weakly polarised surface $(Y,A)$ the \emph{$k$th algebraic capacity} of $(Y,A)$ is
$$\calg_k(Y,A):=\inf_{D\in\Nefqc(Y)}\{D\cdot A:\chi(D)\geq k+\chi(\mO_Y)\}$$
where $\Nefqc(Y)$ denotes the set of nef $\Q$-Cartier $\Z$-divisors on $Y$.
\end{definition*}

The infimum is always attained and, in many situations, agrees with the infimum taken over all nef $\Q$-Cartier $\Q$- or $\R$-divisors on $Y$. Some natural questions are the following:
\begin{itemize}
\item how does $\calg_k(Y,A)$ vary as $A$ moves in the big cone of $Y$?
\item if $\wt{Y}\to Y$ is a resolution of singularities, how are $\calg_k(Y,A)$ and $\calg_k(\wt{Y},\wt{A})$ related?
\item what are the (sub-leading) asymptotics of $\calg_k(Y,A)$ as $k\to\infty$?
\item what do optimisers for $\calg_k(Y,A)$ look like?
\end{itemize}

We will give answers to each of these as part of the process of building a framework for algebraic capacities, which we summarise here.

\begin{thm*}[Prop.~\ref{prop:chamber} + Cor.~\ref{cor:cap_cont}] Suppose $Y$ is a projective $\Q$-factorial surface. For each $k$ the function $\calg_k(Y,\cdot)\colon A\mapsto\calg_k(Y,A)$ is continuous on the big cone of $Y$. Moreover, there is a locally finite chamber decomposition with respect to which $\calg_k(Y,\cdot)$ is piecewise-linear. In \S\ref{sec:degen} we describe how to relate the values of $\calg_k(Y,\cdot)$ on certain walls of the nef cone to algebraic capacities on the corresponding contractions of $Y$.
\end{thm*}

\begin{thm*}[Thm.~\ref{thm:smooth_asy} + Prop.~\ref{prop:sing_asy}] \label{thm:asy1} Suppose $(Y,A)$ is a pseudo-polarised surface with $Y$ either smooth or toric. Then
$$\lim_{k\to\infty}\frac{\calg_k(Y,A)^2}{k}=2A^2$$
\end{thm*}

That is, $\calg_k(Y,A)$ grows like $\sqrt{2A^2k}$. We define error terms
$$\ealg_k(Y,A):=\calg_k(Y,A)-\sqrt{2A^2k}$$

\begin{thm*}[Thm.~\ref{thm:lims_z}] \label{thm:asy2} Suppose $(Y,A)$ is a pseudo-polarised surface with $Y$ either smooth or toric. If $A=qA_0$ for some $q\in\R_{>0}$ and some big and nef $\Z$-divisor $A_0$ then
$$\limsup_{k\to\infty}\ealg_k(Y,A)=\on{gap}(Y,A)+\frac{1}{2}K_Y\cdot A\text{ and }\liminf_{k\to\infty}\ealg_k(Y,A)=\frac{1}{2}K_Y\cdot A$$
for some constant $\on{gap}(Y,A)$ associated to the pair $(Y,A)$.
\end{thm*}

Note that $-K_Y$ is effective for all toric surfaces. Intuitively one might expect optimisers for $\calg_k(Y,A)$ when $k$ is large to `resemble' large multiples of $A$; for instance, one can view the optimisation problem defining algebraic capacities as a sort of isoperimetric problem. We formalise this intuition as part of the proof of Thm.~\ref{thm:asy1} and Thm.~\ref{thm:asy2}.

For each polarised surface $(Y,A)$ there is a kind of `Frobenius problem' (c.f.~\cite{bdr}) asking the following: what do the sets
$$\{\calg_k(Y,A):k\geq k_0\}$$
look like as $k_0$ becomes large? If $A$ is a $\Z$-divisor, we say that $(Y,A)$ is \emph{tightly-constrained} (c.f.~\cite[Def.~5.3]{bwo}) if there exists $k_0$ such that
$$\{\calg_k(Y,A):k\geq k_0\}=\Z_{\geq x}$$
for some $x\in\Z_{\geq0}$.

\begin{thm*}[Cor.~\ref{cor:alg_tc}] If $Y$ is smooth and $A$ is a big and nef $\Z$-divisor, then $(Y,A)$ is tightly-constrained if and only if
\begin{equation} \label{eqn:intro_gap} \tag{$\ast$}
\inf_{D\in N^1(Y)}\{D\cdot A:D\cdot A>0\}=1
\end{equation}
If $Y$ is toric, then $(Y,A)$ is tightly-constrained if and only if $A$ is a primitive Cartier divisor.
\end{thm*}

The infimum (\ref{eqn:intro_gap}) is equal to the constant $\on{gap}(Y,A)$ from Thm.~\ref{thm:asy2}. Finally, we find explicit eventual expressions for the counting function for algebraic capacities: the \emph{cap function}
$$\on{cap}_{(Y,A)}(x):=\#\{k:\calg_k(Y,A)\leq x\}$$

\begin{thm*}[Prop.~\ref{prop:cap_rep} + Prop.~\ref{prop:low_bound}] Suppose that $(Y,A)$ is a pseudo-polarised surface with $Y$ either smooth or toric. Suppose that $A$ is a Cartier divisor on $Y$. Then there exist $\gamma_0,\dots,\gamma_{A^2-1}\in\Q$ and $x_0\in\Z_{\geq0}$ such that for $x\in\Z_{\geq x_0}$
$$\on{cap}_{(Y,A)}(x)=\frac{1}{2A^2}x^2+\frac{-K_Y\cdot A}{2A^2}x+\gamma_i$$
where $x\equiv i\on{mod}{A^2}$. Moreover, we find a computable value for $x_0$.
\end{thm*}

This also expresses the cap function of $(Y,A)$ as a kind of multigraded Hilbert function. We expect there to be generalisations of most of these results to other mildly singular surfaces such as surfaces with rational singularities (including orbifolds).

While the remainder of the introduction will focus on applications of algebraic capacities to problems in other fields, we comment that algebraic capacities are intriguing invariants in their own right, and that we anticipate they will interestingly relate to other quantities in algebraic positivity. For example, in \cite{cw} $\calg_1(Y,A)$ is shown to upper bound the Gromov width of $Y$ in many situations and hence from \cite[Cor.~2.1.D]{mcp} we see that the Seshadri constant $\eps(Y,A)$ \cite[\S5.1]{laz1} is often bounded above by $\calg_1(Y,A)$.

\subsection{Algebraic capacities and symplectic embeddings}

We outline the connections of algebraic capacities to Embedded Contact Homology (ECH) and symplectic embeddings. ECH is a homology theory associated to any symplectic $4$-manifold $(X,\omega)$ that produces an increasing sequence of real numbers
$$\cech_k(X,\omega)$$
called the \emph{ECH capacities} of $(X,\omega)$. They have, among other properties, $\cech_k(X,\omega)\leq\cech_k(X',\omega')$ when there is a symplectic embedding $(X,\omega)\to (X',\omega')$. By a symplectic embedding we mean a smooth embedding $\iota\colon X\to X'$ such that $\iota^*\omega'=\omega$. In this sense the ECH capacities obstruct symplectic embeddings and in some cases these obstructions are sharp \cite{mcd,cg}.

We state the main connections between algebraic capacities and ECH from \cite{bwo,cw}.

\begin{thm*}[{\cite[Thm.~1.5]{bwo} + \cite[Thm.~1.3]{cw}}] \label{thm:alg_ech} Suppose $(X,\omega)$ is a star-shaped domain such that its interior symplectically embeds into a polarised smooth or toric surface $(Y,A)$ regarded as a symplectic manifold (or orbifold) with symplectic form Poincar\'e dual to $A$. Then
$$\cech_k(X,\omega)\leq\calg_k(Y,A)$$
If $X_\Omega$ is a rational-sloped convex toric domain (see Def.~\ref{def:conv_dom}) and $(Y_\Omega,A_\Omega)$ is the polarised toric surface corresponding to $\Omega$ then
$$\cech_k(X_\Omega)=\calg_k(Y,A)$$
\end{thm*}

Some of the main value of these connections that we will capitalise on in the present work is that there are often additional computational tools available to study algebraic capacities, especially in situations where the nef cone is well-behaved \cite{hk}.

The analog to Thm.~\ref{thm:asy1} in the symplectic context is known as the `Weyl law for ECH' \cite[Thm.~1.1]{asy1}. As a consequence of Thm.~\ref{thm:asy1} and Thm.~\ref{thm:alg_ech} we recover the Weyl law for rational-sloped convex toric domains using algebraic methods.

A more novel application that algebraic capacities offer to ECH capacities is in studying their sub-leading asymptotics. Using the Weyl law for ECH one defines error terms by
$$e_k(X,\omega):=\cech_k(X,\omega)-\sqrt{4\on{vol}(X,\omega)k}$$
It is believed that $e_k(X,\omega)$ is $O(1)$ for all compact symplectic $4$-manifolds, though the best known bounds are of the form $O(k^{1/4})$. We can use the previous results on the sub-leading asymptotics of algebraic capacities to locate an optimal bound for many toric domains.

\begin{thm*}[Prop.~\ref{prop:lims}] \label{thm:intro_subl} Suppose $X_\Omega$ is a rational-sloped convex toric domain. If $\Omega=q\Omega_0$ for some $q\in\R_{>0}$ and some primitive lattice convex domain $\Omega_0$ (see Def.~\ref{def:conv_dom}), then
$$\limsup_{k\to\infty}e_k(X_\Omega)=q+\frac{1}{2}\ell_{\Z^2}(\partial\Omega)\text{ and }\liminf_{k\to\infty}e_k(Y,A)=\frac{1}{2}\ell_{\Z^2}(\partial\Omega)$$
where $\ell_{\Z^2}$ is the lattice or affine length.
\end{thm*}

Hutchings conjectured \cite[Conj.~1.5]{hu} that $e_k(X_\Omega)$ converges for `generic' $\Omega$ -- the cases we address are among those that are non-generic.

We also resolve two conjectures \cite[Conj.~1.4 + Conj.~5.7]{bwo} on quasi-polynomial presentations of the \emph{cap function}
$$\on{cap}_{(X,\omega)}(x):=\#\{k:\cech_k(X,\omega)\leq x\}$$
Define the Ehrhart function $\on{ehr}_\Lambda(x)$ of a polytope $\Lambda\subseteq\R^n$ as
$$\on{ehr}_\Omega(x):=\# x\Lambda\cap\Z^n$$
for $x\in\Z_{\geq0}$. When $\Lambda$ is a rational polytope this is a quasi-polynomial \cite{ehr}.

\begin{thm*}[Prop.~\ref{prop:tc_conj} + Thm.~\ref{thm:prim_ech}] Suppose $\Omega$ is a primitive lattice convex toric domain with $\Omega$-perimeter $\lambda$ (see \S\ref{sec:tor_ech}). Then there exist $\gamma_0,\dots,\gamma_{\lambda-1}\in\Q$ and $x_0\in\Z_{\geq0}$ such that
\begin{align*} \on{cap}_{X_\Omega}(i+\lambda x)&=\on{ehr}_\Omega(x)+\gamma_i \\
&=h^0(Y_\Omega,xA_\Omega)+\gamma_i
\end{align*}
when $x\in\Z_{\geq x_0}$ and $x\equiv i\on{mod}{\lambda}$. That is, for $x\in\Z_{\geq x_0}$
$$\on{cap}_{X_\Omega}(x)=\frac{1}{4\on{vol}(\Omega)}x^2+\frac{\ell_{\Z^2}(\partial\Omega)}{4\on{vol}(\Omega)}x+\gamma_i$$
This allows us to find quasi-polynomial presentations for the cap function of any rational convex toric domain. Moreover, we find a computable value for $x_0$.
\end{thm*}

This result helps understand conjectures on the structure of ellipsoid embeddings made by Cristofaro-Gardiner--Holm--Mandini--Pires \cite[\S6]{chmp}. Finally, we observe that Thm.~\ref{thm:asy2} combined with Thm.~\ref{thm:alg_ech} supplies (asymptotic) upper bounds for $e_k(X,\omega)$ for many $(X,\omega)$ and we deduce in Cor.~\ref{cor:ek_obs} new embedding obstructions between many toric domains of the same volume. 

\subsection{Algebraic capacities and minimal hypersurfaces} We outline the connection of algebraic capacities with the theory of minimal hypersurfaces. 

Studying minimal hypersurfaces in Riemannian manifolds is a fundamental problem in differential geometry. For a Riemannian manifold $(M,g)$ one can define a sequence of increasing real numbers called \emph{codimension $c$ min-max widths}
$$\omega_p^c(M,g)$$
measuring the area of certain `$p$-sweepouts' -- a formal version of families of codimension $c$ submanifolds in $M$. The min-max widths satisfy a Weyl law and have many properties and types of applications in common with ECH; for example, \cite{song}.

\begin{prop*}[Ex.~\ref{ex:min_hyp}] Suppose $M$ is a smooth complex projective algebraic variety equipped with an ample divisor $A$. Let $g$ be the metric on $M$ corresponding to $A$. Then the codimension two min-max widths for $(M,g)$ satisfy
\begin{equation} \label{eqn:hyp_alg} \tag{$\star$}
\omega_p^2(M,g)\leq\calg_p(M,A)
\end{equation}
\end{prop*}

It is conceivable that (\ref{eqn:hyp_alg}) is actually an equality in this situation and that, as in the case of ECH, further ties are present between the theory of minimal hypersurfaces and algebraic capacities. An advantage generally offered by algebraic capacities is that their defining optimisation problems are often tractable when the nef cone is well-behaved (for instance, see \cite{hk}) or it is at least always the case that estimates are readily available.

\subsection{Future directions}

There are several directions in which we intend to continue the threads of this paper, three of which are:
\begin{itemize}
\item \emph{Algebraic capacities in higher (co)dimensions:} Let $Y$ be a mildly singular variety of dimension $n$. Choosing a curve class $C\in\on{NE}(Y)$ one can define
$$\calg_k(Y,C):=\inf_{D\in\Nefqc(Y)}\{D\cdot C:\chi(D)\geq k+\chi(\mO_Y)\}$$
It would be interesting to study the properties of these invariants, to examine how they change with dimension, and to see if they have any connections to symplectic or differential geometry. There is also a developing theory of nef cycles \cite{delv} of codimension $>1$ and so one could try to formulate similar optimisation problems in that context.
\item \emph{Irrational symplectic manifolds and ind-schemes:} So far the applications of algebraic capacities to symplectic embedding problems have been limited to non-generic cases like rational-sloped convex toric domains. We predict that an extension to cover more generic symplectic manifolds might be accessed by using increasingly accurate `approximations' via non-generic manifolds that our methods can treat, and that the approximations on the algebraic side can be collected into some kind of ind-scheme. Developing a theory of algebraic capacities for such objects is therefore an attractive prospect.
\item \emph{Other phenomenology:} There are many other invariants and structures that have emerged from studying embedding obstructions through ECH -- for instance, the weight sequences from \S\ref{sec:tor_ech} or infinite staircases \cite{mcs,usher,chmp} -- that we would like to explore through the lens of algebraic capacities and see what avatars of these phenomena exist in algebraic geometry.
\end{itemize}

\subsection*{Outline}

In \S\ref{sec:alg_cap_con} we will formally construct algebraic capacities for weakly polarised surfaces and study some of their basic properties. We consider what happens as the polarisation $A$ changes in \S\ref{sec:continuity}. We prove our main results on the asymptotics of algebraic capacities in \S\ref{sec:asymptotics}. At this point we transition to applications of algebraic capacities to symplectic and Riemannian geometry. In \S\ref{sec:ech} we review the relationship between ECH and algebraic capacities, and apply the results of previous sections to analyse the sub-leading asymptotics of ECH capacities for convex toric domains. Lastly, in \S\ref{sec:rie} we concisely describe the connection between algebraic capacities and min-max widths in the theory of minimal hypersurfaces.

\subsection*{Acknowledgements}

I would like to thank Julian Chaidez, David Eisenbud, Michael Hutchings, Dan Cristofaro-Gardiner, Vinicius Ramos, Jonathan Lai, Mengyuan Zhang, Antoine Song, \DJ an Daniel Erdmann-Pham, Tara Holm, and Stefano Filipazzi for many insightful and supportive conversations. I am especially glad for the range of specialisms represented by these researchers, and grateful to each of them for bearing with me as I attempted to translate some aspect of algebraic capacities into their world and back again.

\section{Constructing algebraic capacities} \label{sec:alg_cap_con}

Whenever we refer to a `surface' below we will mean a projective normal algebraic surface over the complex numbers, not necessarily smooth.

\subsection{Divisors}

We will need to take some care at a few points when dealing with $\Q$-divisors and so we spell out the parts of the general treatment we require here. Always $\bK\in\{\Z,\Q,\R\}$. We say that a Weil $\Z$-divisor $D$ on a surface $Y$ is $\Q$\textit{-Cartier} if some integer multiple of $D$ is Cartier. A Weil $\R$-divisor is $\Q$-Cartier if it can be expressed as an $\R$-linear combination of $\Q$-Cartier $\Z$-divisors. $Y$ is said to be $\Q$\textit{-factorial} if every Weil $\Z$-divisor on $Y$ is $\Q$-Cartier. For instance, every toric surface is $\Q$-factorial. We fix notation:
\begin{itemize}
\item Denote by $\on{WDiv}(Y)_\bK$ the Weil divisors on a $\Q$-factorial surface $Y$ with coefficients in $\bK$
\item Denote by $\on{NS}(Y)_\bK:=\on{NS}(Y)\otimes_\Z\bK$ the group of integral Weil divisors on $Y$ up to algebraic equivalence -- the N\'eron--Severi group -- tensored with $\bK$
\item Denote by $\on{nef}(Y)_\bK$ the Weil divisors in $\on{WDiv}(Y)_\bK$ that are nef
\item Denote by $\on{Nef}(Y)_\bK$ the divisor classes in $\on{NS}(Y)_\bK$ corresponding to nef divisors
\item Denote by $\nefqc(Y)_\bK$ the divisor classes in $\on{WDiv}(Y)_\bK$ corresponding to nef $\Q$-Cartier $\bK$-divisors
\item Denote by $\Nefqc(Y)$ the divisor classes in $\on{NS}(Y)_\Z$ corresponding to nef $\Q$-Cartier $\Z$-divisors.
\item Denote by $\on{Big}(Y)$ the divisor classes in $\on{NS}(Y)_\R$ corresponding to big $\R$-divisors.
\end{itemize}

Observe that when $Y$ is $\Q$-factorial we have
$$\nefqc(Y)_\bK=\on{nef}(Y)_\bK\text{ and }\Nefqc(Y)_\bK=\on{Nef}(Y)_\bK$$
For $D\in\on{WDiv}(Y)_\R$ denote by $\lf D\rf$ the `round-down' of $D$ defined by
$$\lf\sum a_iD_i\rf:=\sum\lf a_i\rf D_i$$
where $D_i$ are prime Weil divisors on $Y$.

Lastly, when $Y$ is toric we denote the polytope corresponding to a nef torus-invariant divisor $D$ by $P(D)$. We recall that the lattice points in $P(D)$ are in bijection with a basis for $H^0(D)$. When $D$ is not nef, instead of a polytope we obtain an oriented hyperplane arrangement $\mA(D)$ as in \cite{perl}. Similarly, the lattice points in the `positive region' bounded by $\mA(D)$ index a basis for global sections \cite[\S9.1]{cls}.

\subsection{Construction of algebraic capacities}

We formally define the invariants coming from the optimisation problems in \cite{bwo,cw,bwt}.

\begin{definition} \label{def:alg_cap} For a weakly polarised surface $(Y,A)$ define the $k$\emph{th algebraic capacity} of $(Y,A)$ to be
$$\calg_k(Y,A):=\inf_{D\in\Nefqc(Y)}\{D\cdot A:\chi(D)\geq k+\chi(\mO_Y)\}$$
\end{definition}

\begin{lemma} \label{lem:alg_attain} Suppose $(Y,A)$ is a weakly polarised surface. The infimum defining $\calg_k(Y,A)$ is achieved by a nef $\Z$-divisor.
\end{lemma}

\begin{proof} This follows essentially from Kleiman's criterion for nefness, from which it is apparent that the region of the real nef cone satisfying $D\cdot A\leq r$ is compact for any $r\geq0$. Pick a $\Z$-divisor $D_0$ with $\chi(D_0)\geq k+\chi(\mO_Y)$ and observe that $\calg_k(Y,A)\leq D_0\cdot A$. This implies that the infimum defining $\calg_k(Y,A)$ ranges over nef $\Z$-divisors $D$ with $D\cdot A\leq D_0\cdot A$ of which there are finitely many.
\end{proof}

We define the \emph{cap function} of $(Y,A)$ to be the counting function for $\calg_k(Y,A)$; that is,
\begin{align*}
\on{cap}_{(Y,A)}(x)&:=\#\{k:\calg_k(Y,A)\leq x\} \\
&=1+\sup_{D\in\Nefqc(Y)}\{\chi(D)-\chi(\mO_Y):D\cdot A\leq x\}
\end{align*}

Define the \emph{index} of a $\Q$-Cartier $\Z$-divisor $D$ on $Y$ by
$$I(D):=D\cdot(D-K_Y)$$
In a situation where Noether's formula holds -- for instance when $Y$ is smooth or has at worst canonical singularities \cite{ypg} -- we have
$$\chi(D)=\chi(\mO_Y)+\frac{1}{2}I(D)$$
and so in such a situation
$$\calg_k(Y,A)=\inf_{D\in\Nefqc(Y)}\{D\cdot A:I(D)\geq 2k\}$$
We briefly discuss how this relates to optimisation problems in symplectic geometry. The index $I(D)$ agrees with the ECH index \cite[Def.~2.14]{hu3} of the homology class of $D$ in $H_2(Y,\Z)$ and the intersection product $D\cdot A$ is the symplectic area of $D$ with respect to the Poincar\'e dual of $A$. 

\subsection{Vanishing theorems}

Certain vanishing theorems give simpler presentations of algebraic capacities. We mention two that have found use in obstructing symplectic embeddings \cite{cw,bwo}.

\begin{lemma}[{\cite[Thm. 9.3.5.]{cls}}] \label{lem:nefq_van} Suppose $Y$ is a toric surface and $D$ is a nef $\Q$-divisor. Then
$$h^p(D):=h^p(\lf D\rf)=0\text{\textnormal{ for all $p>0$}}$$ 
\end{lemma}

\begin{prop} \label{prop:toricq} Let $Y$ be a toric surface, and let $A$ be a big $\R$-divisor on $Y$. Then
$$\calg_k(Y,A)=\inf_{D\in\on{Nef}(Y)_\Z}\{D\cdot A:h^0(D)\geq k+1\}$$
for all $k\geq0$.
\end{prop}

\begin{proof} This follows immediately from Lem.~\ref{lem:nefq_van} and the fact that $Y$ is rational.
\end{proof}

This formulation of algebraic capacities was used in the author's original paper \cite{bwo} that initiated the use of this kind of algebraic optimisation problem to study symplectic embeddings.

Kawamata--Viehweg vanishing gives a similar reformulation when $-K_Y$ is nef and big.

\begin{prop} \label{prop:cap_kv} Let $Y$ be a smooth surface such that $-K_Y$ is nef and big, and let $A$ be a big $\R$-divisor on $Y$. Then
$$\calg_k(Y,A)=\inf_{D\in\on{Nef}(Y)_\Z}\{D\cdot A:h^0(D)\geq k+1\}$$
for all $k\geq0$.
\end{prop}

\subsection{Preferable divisors}

We codify what it means for a divisor to be `preferable' to another from the point of view of the optimisation problems in Definition \ref{def:alg_cap} for pseudo-polarised surfaces.

\begin{definition}[{c.f.~\cite[Def.~3.11]{cw}}] \label{def:pref} Let $Y$ be a projective normal surface.
\begin{itemize}
\item We say that a Weil $\Q$-Cartier $\R$-divisor $D$ on $Y$ is \textit{index-preferable} to another Weil $\Q$-Cartier $\R$-divisor $D'$ on $Y$ if $\chi(D)\geq\chi(D')$.
\item We say that $D$ is \textit{area-preferable} to $D'$ if $D\cdot A\leq D'\cdot A$ for all big and nef $\Q$-Cartier $\R$-divisors $A$ on $Y$. 
\item We say that $D$ is \textit{preferable} to $D'$ if $D$ is both index-preferable and area-preferable to $D'$.
\end{itemize}
\end{definition}

Observe that $D_0$ is area-preferable to $D$ if and only if $D-D_0$ is effective. When $Y$ is smooth one can replace the inequality $\chi(D)\geq\chi(D')$ with $I(D)\geq I(D')$. Our first use for this notion is to show the following.

\begin{prop} \label{prop:nef_int} Suppose $Y$ is a smooth surface with $-K_Y$ effective and that $A$ is a big and nef $\R$-divisor on $Y$. Then, for all $k>\on{max}\{0,-\chi(\mO_Y)\}$
$$\calg_k(Y,A)=\inf_{D\in\on{nef}(Y)_\Q}\{D\cdot A:I(D)\geq 2k\}$$
That is, for every nef $\Q$-divisor with $I(D)\geq 2k$ there is a preferable nef $\Z$-divisor. The same conclusion holds for all $k>0$ when $Y$ is a possibly singular toric surface.
\end{prop}

We can consider optimisation problems obtained ranging over $\Q$-divisors (or $\R$-divisors) instead of $\Z$-divisors as in Def.~\ref{def:alg_cap}. Set
$$\calg_k(Y,A)_\Q:=\underset{D\in\nefqc(Y)_\Q}{\on{min}}\{D\cdot A:I(D)\geq 2k\}$$
where $\chi(D):=\chi(\lf D\rf)$ for $D$ a Weil $\Q$-divisor. Observe that it is necessary to work with Weil $\Q$-divisors as opposed to $\Q$-divisor classes since if $D$ and $D'$ are $\Q$-algebraically equivalent it does not imply that the round-downs $\lf D\rf$ and $\lf D'\rf$ are $\Z$-algebraically (or even numerically) equivalent. This infimum yields the same result as when ranging over nef $\R$-divisors.

\begin{cor} When $(Y,A)$ is a pseudo-polarised surface with either $Y$ smooth and $-K_Y$ effective or $Y$ toric, then
$$\calg_k(Y,A)=\calg_k(Y,A)_\Q$$
for all $k>\on{max}\{0,-\chi(\mO_Y)\}$.
\end{cor}

This shows that the construction using nef $\Q$-divisors in \cite[Thm.~1.5]{bwo} for toric surfaces agrees with our construction here. Our main tool for locating preferable divisors is the `isoparametric transform' of \cite{phys}. This takes an effective divisor $D$ to
$$\ip_Y(D):=D-\sum_{D\cdot D_i<0}\lc\frac{D\cdot D_i}{D_i^2}\rc D_i$$
where the sum ranges over all prime divisors $D_i$ with $D\cdot D_i<0$. Observe that such $D_i$ must have $D_i^2<0$ and hence the coefficients in the sum are positive. We denote by $\ip^n_Y$ the result of iterating $\ip_Y$ $n$ times. In \cite{phys} the following is shown.

\begin{thm}[{\cite[Thm.~1.1 + 1.2]{phys}}] \label{thm:phys} For any effective divisor $D$ on a smooth surface $Y$ we have
$$h^0(D)=h^0(\ip_Y(D))$$
After a finite number $n$ of iterations $\ip^n_Y(D)$ lies in $\on{Nef}(Y)_\Z$.
\end{thm}

The isoparametric transform stabilises once it has reached a nef divisor; we denote this limit by $\ip_Y^\infty(D)$. We show that the isoparametric transform yields index-preferable divisors.

\begin{lemma}[{c.f.~\cite[Lem.~3.12]{cw}}] \label{lem:iso_pref} Let $Y$ be a smooth surface with $D$ an effective $\Z$-divisor on $Y$. Suppose $C_1,\dots,C_n$ is a collection of curves intersecting $D$ negatively. Then either one of the $C_i$ is a $(-1)$-curve or
$$I(D')\geq I(D)$$
where
$$D'=D-\sum_{i=1}^n\lc\frac{D\cdot C_i}{C_i^2}\rc C_i$$
In particular,
$$I(\ip_Y(D))\geq I(D)$$
if no $(-1)$-curve intersects $D$ negatively.
\end{lemma}

\begin{proof} Suppose $n=1$ so that there is only one curve $C$. If $C^2=-1$ we are done, so let $C^2=-r$ for $r\geq2$. Let $D\cdot C=-\ell$ so that
$$D'=D-\lc\frac{\ell}{r}\rc C=:D-mC$$
Let $\pi\colon Y\to\ol{Y}$ be the contraction of $C$ to the singular surface $\ol{Y}$. We can compute
\begin{align*}
I(D')&=(D-mC)\cdot(D-mC-K_Y) \\
&=I(D)-2mD\cdot C+(-mC)\cdot(-mC-K_Y) \\
&=I(D)+2m\ell+(-mC)\cdot(-mC-\pi^*K_{\ol{Y}}-\frac{2-r}{r}C) \\
&=I(D)+2m\ell-m^2r-(2-r)m
\end{align*}
Observe that $1>m-\frac{\ell}{r}\geq 0$ by definition and so $\ell>r(m-1)$ or equivalently $\ell+r>rm$. Continuing:
\begin{align*}
I(D')&=I(D)+2m\ell-mr(m+\frac{2-r}{r}) \\
&>I(D)+2m\ell-(\ell+r)(m+\frac{2-r}{r}) \\
&=I(D)+m\ell+\ell\cdot\frac{r-2}{r}-mr+r-2 \\
&\geq I(D)+m\ell-r(m-1)-2 \\
&> I(D)+(m-1)\ell-2 \\
&\geq I(D)-2
\end{align*}
using $r\geq 2$ and $m\geq 1$. However $I(\cdot)$ is even and so $I(D')>I(D)-2$ implies that $I(D')\geq I(D)$.

Now induct on the number of curves. Suppose the formula holds for a set of $n$ curves meeting an effective divisor negatively. Suppose curves $C_1,\dots,C_n,C$ intersect $D$ negatively. If any of the curves is a $(-1)$-curve then we are done. Assume not. Notate
$$D\cdot C=-\ell,\;\;\; C^2=-r,\;\;\;\lc\frac{D\cdot C}{C^2}\rc=m$$
and
$$F=\sum_{i=1}^{n-1}m_iC_i$$
so that $D'=D-F-mC$. Compute
\begin{align*}
I(D-F-mC)&= \\
&=I(D-F)+2mF\cdot C-2mD\cdot C+I(-mC) \\
&\geq I(D-F)+2m\ell-mr(m+\frac{2-r}{r}) \\
&>I(D-F)+(m-1)\ell-2 \\
&\geq I(D-F)-2
\end{align*}
where we used that $F\cdot C\geq0$ since $F$ is effective and supported away from $C$. By inductive assumption we have $I(D-F)\geq I(D)$ and so we have $I(D')>I(D)-2$. Since $I(\cdot)$ is even we can conclude that $I(D')\geq I(D)$ as desired.
\end{proof}

We use Lem.~\ref{lem:iso_pref} to prove the following result related to \cite[Thm.~3.3]{cw}.

\begin{prop} \label{prop:ne_nef} Suppose $Y$ is a smooth surface with a big and nef $\R$-divisor $A$ on $Y$. Then
$$\inf_{D\in\on{NE}(Y)_\Z}\{D\cdot A:I(D)\geq 2k\}=\inf_{D\in\on{Nef}(Y)_\Z}\{D\cdot A:I(D)\geq 2k\}$$
\end{prop}

\begin{proof} We prove this by induction on the number of blowups $b$ relating $Y$ to a minimal surface. When $b=0$ -- and so $Y$ is minimal -- $Y$ has no $(-1)$-curves and so $\ip_Y(D)$ is preferable to $D$ for any $D\in\on{NE}(Y)_\Z$ by Lem.~\ref{lem:iso_pref}. Iterating $\ip_Y$ gives that $\ip_Y^\infty(D)$ is a preferable nef $\Z$-divisor to $D$ from which the result follows.

Suppose $b>0$, and assume that the statement holds for all smooth surfaces expressible at most $b-1$ blowups from a minimal surface. Let $Y$ be expressed as $b$ blowups of a minimal surface. Since $b>0$ we see that $Y$ is not minimal and so contains a $(-1)$-curve. Let $D\in\on{NE}(Y)_\Z$. If $D$ intersects a $(-1)$-curve $E$ nonpositively then we can write $D=\pi^*\ol{D}+mE$ for some $m\geq0$ and some effective $\Z$-divisor $\ol{D}$ on $\ol{Y}$, where $\pi\colon Y\to\ol{Y}$ is the contraction of $E$. By the inductive hypothesis there exists a preferable nef $\Z$-divisor $\ol{D}_0$ to $\ol{D}$. We claim that $\pi^*\ol{D}_0$ is a preferable nef $\Z$-divisor to $D$. Indeed, we can compute
$$D\cdot A\geq\pi^*\ol{D}\cdot A\geq\pi^*\ol{D}_0\cdot A$$
and
$$I(D)\leq I(\pi^*\ol{D})=I(\ol{D})\leq I(\ol{D}_0)=I(\pi^*\ol{D}_0)$$
If $D$ intersects all $(-1)$-curves positively, consider $\ip_Y(D)$. This is a preferable effective $\Z$-divisor to $D$ by Lem.~\ref{lem:iso_pref} and by noting that $D-\ip_Y(D)$ is effective. If $\ip_Y(D)$ intersects a $(-1)$-curve nonpositively, we are done by the argument above. If not, apply $\ip_Y(D)$ again. Continuing in this way, we reach a preferable nef $\Z$-divisor if there exists $n$ such that $\ip_Y^n(D)\cdot E\leq0$ for some $(-1)$-curve $E$, or we reach $\ip_Y^\infty(D)$, which is nef by Thm.~\ref{thm:phys} and preferable by Lem.~\ref{lem:iso_pref}.
\end{proof}

\begin{proof}[Proof of Prop.~\ref{prop:nef_int}] Consider a nef $\Q$-divisor $D$ with $I(D)\geq 2k$. If $h^0(D)\geq 1$ then $\lf D\rf$ is an effective $\Z$-divisor that is preferable to $D$. The result then follows by Prop.~\ref{prop:ne_nef}. For $k>-\chi(\mO_Y)$, if $h^2(\lf D\rf)=0$ then we see that $h^0(\lf D\rf)\geq 1$.

We claim that if $-K_Y$ is effective then $\lf D\rf$ must be effective whenever $D$ is a nef $\Q$-divisor with $I(D)\geq 2k$ for any $k>0$. Consider the adjoint divisor $K_Y-\lf D\rf=K_Y-\lc D\rc+\Delta$ where $\Delta$ is an effective divisor with all nonzero coefficients equal to $1$. Observe that $\lc D\rc$ is an effective $\Z$-divisor since $D$ is nef. This adjoint divisor takes the form $\Delta-\text{effective divisor}$ and so is either noneffective or trivial. If noneffective we are done from $h^2(D)=h^0(K_Y-\lc D\rc+\Delta)=0$, which leaves the case when the adjoint divisor is trivial. In that case $\lf D\rf=K_Y$ and so $\chi(D)=\chi(\mO_Y)$ but by the assumption $k>0$ we have $\chi(D)>\chi(\mO_Y)$ giving a contradiction.

The previous argument covers smooth toric surfaces. For singular toric surfaces we can use a combinatorial method to achieve the same result, which we now outline. For nef $\Q$-divisors $\chi(D)=h^0(D)=L_{P(D)}:=\#P(D)\cap\Z^2$. Consider the $\Z$-divisor $\lf D\rf$ whose hyperplane arrangement is obtained by translating the facet hyperplanes of $P(D)$ along their inward normal directions until they include a lattice point. This divisor may not be nef; combinatorially, there may be configurations of edges in $P(D)$ such as that pictured in Fig.~\ref{fig:floor}(a) so that $\mA(\lf D\rf)$ looks like Fig.~\ref{fig:floor}(b). Let $E$ be the divisor corresponding to the hyperplane in $\mA(\lf D\rf)$ that does not meet the positive region as shown in Fig.~\ref{fig:floor}(b). By subtracting $mE$ where $m$ is the smallest positive integer such that the corresponding translated hyperplane meets $\mA(\lf D\rf)$ we obtain the polytope depicted in Fig.~\ref{fig:floor}(c). This clearly does not affect the index and reduces the area, hence producing a preferable $\Z$-divisor. Repeating this process -- and iterating if necessary -- produces a preferable nef $\Z$-divisor to $D$.
\end{proof}

\begin{figure}[h]
\caption{$\lf D\rf$ not nef}
\label{fig:floor}
\begin{center}
\begin{tikzpicture}
\foreach \i in {0,...,2}
	\foreach \j in {0,...,3}
		{ \node (\i \j) at (\i,\j){\tiny $\bullet$};}

\node (a) at (0,2.4){};
\node (b) at (1.2,3){};
\node (c) at (1.6,2.8){};
\node (d) at (1.8,0){};

\draw (a.center) to (b.center);
\draw (c.center) to (d.center);
\draw (1,3) to (2,2);

\node (l) at (1,-0.5){$(a)$}; 

\foreach \i in {0,...,2}
	\foreach \j in {0,...,3}
		{ \node (\i \j) at (\i+4,\j){\tiny $\bullet$};}

\node (a) at (4+0,2){};
\node (b) at (4+1.2,2.6){};
\node (c) at (4+0.95,2.8){};
\node (d) at (4+1.15,0){};

\draw (a.center) to (b.center);
\draw (c.center) to (d.center);
\draw (5,3) to (6,2);

\node (l') at (5,-0.5){$(b)$};

\foreach \i in {0,...,2}
	\foreach \j in {0,...,3}
		{ \node (\i \j) at (\i+8,\j){\tiny $\bullet$};}

\node (a) at (8+0,2){};
\node (b) at (8+1.2,2.6){};
\node (c) at (8+0.95,2.8){};
\node (d) at (8+1.15,0){};

\draw (a.center) to (b.center);
\draw (c.center) to (d.center);
\draw (8,3) to (10,1);

\node (l'') at (9,-0.5){$(c)$};
\end{tikzpicture}
\end{center}
\end{figure}

\subsection{Properties} \label{sec:properties}

We list some of the essential properties of algebraic capacities, noticeably parallelling properties for symplectic capacities \cite[\S2]{qsg}.

\begin{prop} \label{prop:cap_prop} The algebraic capacities $\calg_k(Y,A)$ satisfy the following properties:
\begin{itemize}
\item (Conformality) For $q\in\R_{>0}$, $\calg_k(Y,qA)=q\cdot\calg_k(Y,A)$
\item (Disjoint Union) For two polarised surfaces $(Y_1,A_1)$ and $(Y_2,A_2)$ we have
$$\calg_k(Y_1\amalg Y_2,A_1\amalg A_2)=\on{max}\{\calg_{k_1}(Y_1,A_1)+\calg_{k_2}(Y_2,A_2):k_1+k_2=k\}$$
\item (Zero) $\calg_0(Y,A)=0$.
\end{itemize}
\end{prop}

The proofs are purely numerical in nature and so we omit them. Perhaps the most important property in the setting of symplectic geometry is `monotonicity', where symplectic capacities are required to increase with embeddings. Another way of saying this is that symplectic capacities are functors from a category of symplectic manifolds with morphisms given by symplectic embeddings into the poset $(\R,\leq)$. An analog of monotonicity for algebraic capacities of toric surfaces is developed in \cite[\S4]{cw}.

\section{Continuity for algebraic capacities} \label{sec:continuity}

\subsection{Chamber decompositions and continuity on the big cone} \label{sec:chamber}

\begin{prop} \label{prop:chamber} Let $Y$ be a projective $\Q$-factorial surface. For each $k\in\Z_{\geq0}$ there is a locally finite chamber decomposition of the big cone of $Y$
$$\on{Big}(Y)=\bigcup_{D}\ol{\mfk{C}}_D$$
such that
$$\mfk{C}_D=\{A\in\on{Big}(Y):\text{$D$ is the unique optimiser for $\calg_k(Y,A)$}\}\subseteq\on{Big}(Y)$$
so that the chambers are indexed by optimal divisors for $\calg_k(Y,A)$ as $A$ varies.
\end{prop}

\begin{proof} Let $A$ be a generic big $\R$-divisor such that $\calg_k(Y,A)$ has a unique optimal nef $\Z$-divisor $D_0$. There exists an open neighbourhood $U$ of $A$ such that $D_0$ is optimal for $\calg_k(Y,A')$ for all $A'\in U$. Indeed, suppose there is no such neighbourhood. Then we can find a sequence $A_n\to A$ such that $\calg_k(Y,A_n)<D_0\cdot A_n$ for all $n$. Let $D_n$ be a $\Z$-divisor optimising $\calg_k(Y,A_n)$. Pick $D'$ such that $I(D')\geq 2k$. All the divisors $D_n$ will lie in the compact region of the nef cone $\{D\in\on{Nef}(Y):D\cdot A_n\leq M\}$ where $M=\sup\{D'\cdot A_n:n\in\Z_{\geq0}\}$. There are hence finitely many distinct divisors $D_n$ and so the sequence has a subsequence $D_{n_i}$ that stabilises at some $\Z$-divisor $D_*$. It follows that $\calg_k(Y,A_{n_i})=D_*\cdot A_{n_i}$, which converges to $D_*\cdot A$ as $i\to\infty$. We have $D_0\cdot A=\calg_k(Y,A)\leq D_*\cdot A$ and $D_*\cdot A_{n_i}=\calg_k(Y,A_{n_i})<D_0\cdot A_{n_i}$ so that $D_*\cdot A\leq D_0\cdot A$. Hence $D_*\cdot A=D_0\cdot A=\calg_k(Y,A)$ and so $D_0=D_*$. This contradicts our construction of the $A_n$.

For local finiteness, we will show that the chamber structure is finite inside each closed set $R_{A_0,\delta}$ defined as the subset of the big cone bounded by $\partial R_{A_0,\delta}=\{D+\delta A_0:D\in\partial\on{\ol{NE}}(Y)\}$ where $A_0$ is an ample divisor and $\delta>0$. Note that the chambers $\mfk{C}_D$ are cones since if a divisor $D$ is optimal for $\calg_k(Y,A)$ then it is also optimal for $\calg_k(Y,qA)$ for any $q\in\R_{>0}$. It thus suffices to check local finiteness on a bounded region of the big cone. Pick an ample $\R$-divisor $A_0$ and consider the region $S$ of the big cone consisting of all big $\R$-divisors $A$ in $R_{A_0,\delta}$ such that $1-\eps\leq A\cdot A_0\leq 1+\eps$ for some small $\eps>0$. Pick a nef $\Z$-divisor $D_0$ with $\chi(D_0)\geq k+\chi(\mO_Y)$. There exists $M\in\R_{\geq0}$ such that $A\cdot D_0\leq M$ for all $A\in S$. It follows that $\calg_k(Y,A)\leq M$ for all $A\in S$ and so optimisers for $\calg_k(Y,A)$ when $A\in R_{A_0,\delta}$ must be in the set
$$K=\{D\in\on{Nef}(Y)_\Z:D\cdot A\leq M\text{ for some $A\in R_{A_0,\delta}$}\}$$
$K$ is compact since $R_{A_0,\delta}$ is compact and bounded away from the boundary of the big cone and so there are finitely many $\Z$-divisors in $K$, which implies that there are only finitely many chambers in $R_{A_0,\delta}$.
\end{proof}

\begin{cor} \label{cor:cap_cont} For any projective $\Q$-factorial surface $Y$ we have that $A\mapsto\calg_k(Y,A)$ is continuous as a function $\on{Big}(Y)\to\R$ for each $k\in\Z_{\geq0}$.
\end{cor}

\begin{proof} Write $c(A)=\calg_k(Y,A)$. Inside each chamber $\mfk{C}_D$ we have that $c|_{\mfk{C}_D}(A)=D\cdot A$, which is continuous. A wall
$$\mfk{w}=\ol{\mfk{C}}_D\cap\ol{\mfk{C}}_{D'}$$
separating chambers is given by the locus inside $\on{Big}(Y)$ where $\calg_k(Y,A)$ has a nonunique optimiser -- in this case, where both $D$ and $D'$ are optimal -- and so we see that $c$ remains continuous when restricted to $\mfk{C}_D\cup\mfk{C}_{D'}\cup(\ol{\mfk{C}}_D\cap\ol{\mfk{C}}_{D'})$, which gives the result.
\end{proof}

\begin{example} We illustrate this chamber decomposition when $Y$ is the blowup of $\pr^2$ in a point; that is, the Hirzebruch surface $\bF_1$. We use the $\Z$-basis for $\on{NS}(Y)_\Z$ given by $F,D_\infty$ where $F$ is a fibre class and $D_\infty$ is a curve of self-intersection $1$. The nef cone is generated by $F$ and $D_\infty$ and so we see that the only possible optimisers for $\calg_1(Y,A)$ are $F$ and $D_\infty$. The effective cone of $Y$ is spanned by $F$ and $D_\infty-F$ so let $A=\alpha F+\beta(D_\infty-F)$ be a general big $\R$-divisor. We have
$$F\cdot A=\beta\text{ and }D_\infty\cdot A=\alpha$$
so that $F$ is preferable when $\alpha\geq\beta$ and $D_\infty$ is preferable when $\alpha\leq\beta$. The chamber decomposition for $k=1$ is shown in Fig.~\ref{fig:big1}(a). Similarly, we see that the only possible optimisers for $\calg_2(Y,A)$ are $2F$ and $D_\infty$ giving the chamber decomposition in Fig.~\ref{fig:big1}(b). The three possible optimisers for $\calg_3(Y,A)$ are $3F,D_\infty+F,2D_\infty$ and this induces the chamber decomposition in Fig.~\ref{fig:big1}(c).
\begin{figure}[h]
\caption{Chamber decompositions for $\on{Big}(\bF_1)$}
\label{fig:big1}
\begin{center}
\begin{tikzpicture}
\small
\node (o) at (0,0){$\bullet$};
\node (a) at (0,2){};
\node (b) at (1.7,-1.7){};
\node (c) at (1.7,0){};

\draw (o.center) to (a);
\draw (o.center) to (b);
\draw (o.center) to (c);

\node (lf) at (1,1){$\mfk{C}_F$};
\node (ld) at (1.3,-0.6){$\mfk{C}_{D_\infty}$};

\node (l) at (1,-2){(a)};

\node (o) at (4+0,0){$\bullet$};
\node (a) at (4+0,2){};
\node (b) at (4+1.7,-1.7){};
\node (c) at (4+1.7,1.7){};

\draw (o.center) to (a);
\draw (o.center) to (b);
\draw (o.center) to (c);

\node (lf) at (4+0.7,1.4){$\mfk{C}_{2F}$};
\node (ld) at (4+1.3,0){$\mfk{C}_{D_\infty}$};

\node (l) at (4+1,-2){(b)};

\node (o) at (8+0,0){$\bullet$};
\node (a) at (8+0,2){};
\node (b) at (8+1.7,-1.7){};
\node (c) at (8+1.7,1.7){};
\node (d) at (8+1.7,0){};

\draw (o.center) to (a);
\draw (o.center) to (b);
\draw (o.center) to (c);
\draw (o.center) to (d);

\node (lf) at (8+0.7,1.4){$\mfk{C}_{3F}$};
\node (ld) at (8+1.3,0.4){$\mfk{C}_{D_\infty+F}$};
\node (ld) at (8+1.3,-0.5){$\mfk{C}_{2D_\infty}$};

\node (l) at (8+1,-2){(c)};
\end{tikzpicture}
\end{center}
\end{figure}
\end{example}

\subsection{Continuity on the boundary of the nef cone} \label{sec:degen}

Let $Y$ be a smooth surface with a birational morphism $\pi\colon Y\to\ol{Y}$. We relate the algebraic capacities of the possibly singular surface $\ol{Y}$ equipped with a big and nef $\R$-divisor $\ol{A}$ to the algebraic capacities of $Y$ with the big and nef divisor $\pi^*\ol{A}$. We start with the easiest case.

\begin{prop} \label{prop:blowup} Suppose $Y$ is a smooth surface with $E\subseteq Y$ a $(-1)$-curve. Let $\pi\colon Y\to\ol{Y}$ be the contraction of $E$. Then
$$\calg_k(Y,\pi^*\ol{A})=\calg_k(\ol{Y},\ol{A})$$
\end{prop}

\begin{proof} 
Let $\ol{D}$ be a $\Z$-divisor optimising $\calg_k(\ol{Y},\ol{A})$. As $\pi$ is the contraction of a $(-1)$-curve $\pi^*\ol{D}$ is a $\Z$-divisor and $I(\pi^*\ol{D})=I(\ol{D})\geq 2k$ giving
$$\calg_k(Y,A_\alpha)\leq\pi^*\ol{D}\cdot A_\alpha=\ol{D}\cdot\ol{A}=\calg_k(\ol{Y},\ol{A})$$
Let $D'$ optimise $\calg_k(Y,\pi^*\ol{A})$ and write $D'=\pi^*\ol{D'}-mE$ for some $m\in\Z_{\geq0}$. Then we have $2k\leq I(D')\leq I(\pi^*\ol{D'})$ and so
$$\calg_k(Y,\pi^*\ol{A})=D'\cdot\pi^*\ol{A}=\ol{D'}\cdot\ol{A}\geq\calg_k(\ol{Y},\ol{A})$$
as required.
\end{proof}

Note that this implies that optimisers for $\calg_k(Y,\pi^*\ol{A})$ can be chosen on the face of the nef cone corresponding to $\pi$.

In general, when the contraction $\ol{Y}$ is singular there is additional complexity in relating $\chi(D)$ and $\chi(\ol{D})$ since the index formulation of $\calg_k(\ol{Y},\ol{A})$ that we used is no longer valid. The additional methods available for toric surfaces enable us to extend to singular contractions in the toric context. The main extra fact we use is that for a birational morphism $\pi\colon Y\to\ol{Y}$ and a nef $\Q$-divisor $\ol{D}$ on $\ol{Y}$ we have $\chi(\ol{D})=h^0(\ol{D})=h^0(\pi^*\ol{D})=\chi(\pi^*\ol{D})$ from Demazure vanishing since both $\ol{D}$ and $\pi^*\ol{D}$ are nef.

\begin{prop} \label{prop:degen_toric} Suppose $Y$ is a toric surface. Let $\pi\colon Y\to\ol{Y}$ be a birational toric morphism and let $\ol{A}$ be an big and nef $\R$-divisor on $\ol{Y}$. Then,
$$\calg_k(Y,\pi^*\ol{A})=\calg_k(\ol{Y},\ol{A})$$
\end{prop}

\begin{proof} We will freely use nef $\Q$-divisors in light of Prop.~\ref{prop:nef_int}. The proof proceeds along the same lines as for Prop.~\ref{prop:blowup}. The same method yields
$$\calg_k(Y,\pi^*\ol{A})\leq\calg_k(\ol{Y},\ol{A})$$
For the converse inequality, let $D$ be an optimal nef $\Z$-divisor for $\calg_k(Y,\pi^*\ol{A})$. Write $D=\pi^*\ol{D}-\sum_{i=1}^s m_iE_i$ where $E_i$ are exceptional divisors for $\pi$ and for $m_i\in\Q_{>0}$. Hence
$$\calg_k(Y,\pi^*\ol{A})=D\cdot\pi^*\ol{A}=\ol{D}\cdot\ol{A}$$
We have $\chi(\ol{D})=\chi(\pi^*\ol{D})=h^0(\pi^*\ol{D})$, which is greater than $h^0(D)=\chi(D)$. We obtain $\calg_k(\ol{Y},\ol{A})\leq\ol{D}\cdot\ol{A}=\calg_k(Y,\pi^*\ol{A})$ as desired.
\end{proof}

\section{Asymptotics for algebraic capacities} \label{sec:asymptotics}

We investigate the asymptotics of algebraic capacities inspired by questions from algebraic positivity \cite{laz1,laz2} and from obstructions to symplectic embeddings coming from the asymptotics of ECH capacities \cite{hu}. We formulate the following conjecture and prove it in many cases.

\begin{conjecture}[Algebraic Weyl Law] \label{conj:awl} Suppose $(Y,A)$ is a pseudo-polarised surface with $A$ a $\Q$-Cartier $\R$-divisor. Then
$$\lim_{k\to\infty}\frac{\calg_k(Y,A)^2}{k}=2A^2$$
\end{conjecture}

We call this a `Weyl law' after similar asymptotic results in the contexts of symplectic and Riemannian geometry; see \cite[Thm.~1.1]{asy1} and \cite[\S1.1]{lmn}.

\subsection{Asymptotics for smooth surfaces}

We start by establishing Conj.~\ref{conj:awl} for smooth surfaces.

\begin{thm} \label{thm:smooth_asy} Suppose $(Y,A)$ is a smooth pseudo-polarised surface. Then
$$\lim_{k\to\infty}\frac{\calg_k(Y,A)^2}{k}=2A^2$$
\end{thm}

The key idea is to select a favourable basis to work with. By the Hodge index theorem there exists an $\R$-basis $e_0,e_1,\dots,e_s$ for $\on{NS}(X)_\R$ such that $e_0=A$ and the matrix of the intersection form with respect to this basis is
$$\mat{cccc}
A^2 & 0 & \dots & 0 \\
0 & -r_1 & \dots & 0 \\
\vdots & \vdots & \ddots & 0 \\
0 & 0 & \dots & -r_s
\tam$$
for some $r_i>0$. We will employ this notation throughout this subsection. For our current purposes it will suffice to use the following `asymptotic capacities'
$$\casy_k(Y,A):=\inf_{D\in\on{Nef}(Y)_\R}\{D\cdot A:D\cdot(D-K_Y)\geq 2k\}$$
The only difference here is that we have replaced the index with the intersection of $\R$-divisors. This will allow us to use analytic methods and makes no difference to calculations in the limit.

\begin{prop} \label{prop:lim_asy} For any pseudo-polarised smooth surface $(Y,A)$ we have
$$\lim_{k\to\infty}\frac{\casy_k(Y,A)^2}{k}=2A^2$$
\end{prop}

\begin{proof} Let $-K_Y=cA+\sum_{i=1}^sd_ie_i$. Suppose $D$ is a divisor with $D\cdot(D-K_Y)\geq 2k$ and write $D=aA+\sum_{i=1}^sb_ie_i$ so that
$$D\cdot(D-K_Y)=a(a+c)A^2-\sum_{i=1}^sr_ib_i(b_i+d_i)\text{ and }D\cdot A=aA^2$$
For $D\cdot(D-K_Y)\geq 2k$ we must have
\begin{equation} \tag{$\ddag$} \label{eqn:opt}
a(a+c)\geq\frac{2k+\sum_{i=1}^sr_ib_i(b_i+d_i)}{A^2}
\end{equation}
Observe that $D\cdot A$ is minimised when $a$ is minimised, which occurs when $\sum_{i=1}^sr_ib_i(b_i+d_i)$ is minimised; that is, when $b_i=-d_i/2$. Indeed, for large $k$ we have that the right hand side of (\ref{eqn:opt}) is large and so the quadratic function $a(a+c)$ achieves that bound for as small a value of $a$ as possible. Namely, approximating $a(a+c)$ by $(a-|c|)^2$ and $(a+|c|)^2$ for large $a$ we obtain
$$-|c|+\sqrt{\frac{2k-\sum_{i=1}^s\frac{r_id_i^2}{4}}{A^2}}\leq a\leq|c|+\sqrt{\frac{2k-\sum_{i=1}^s\frac{r_id_i^2}{4}}{A^2}}$$
and so for optimal $D$ we have $a\sim\sqrt{2k/A^2}$. Let $D_k$ be such an optimal divisor. Note that $D_k+\delta$ is nef since $A$ is nef, and that the asymptotics of $A\cdot(D_k+\delta)$ are the same up to an $O(1)$ error as the asymptotics of $A\cdot D_k$ since the $b_i$ are constant. Thus $\casy_k(Y,A)\sim\sqrt{2A^2k}$ and the result follows.
\end{proof}

The proof of Prop.~\ref{prop:lim_asy} shows that an effective divisor computing $\casy_k(Y,A)$ for large $k$ is of the form $D=aA+\delta$ for some $a\gg0$ and where $\delta=-\frac{1}{2}\sum_{i=1}^sd_ie_i$. We also see that $aA$ is a nef $\R$-divisor that is an approximate optimiser for $\casy_k(Y,A)$ with $O(1)$ error. Using Prop.~\ref{prop:lim_asy} it is immediate that Thm.~\ref{thm:smooth_asy} is a consequence of the following lemma.

\begin{lemma} \label{lem:alg_asy} For any pseudo-polarised smooth surface $(Y,A)$ we have
$$\lim_{k\to\infty}\frac{\calg_k(Y,A)^2}{k}=\lim_{k\to\infty}\frac{\casy_k(Y,A)^2}{k}$$
\end{lemma}

\begin{proof} This essentially follows since $D-\lf D\rf$ is small and hence $D\cdot(D-K_Y)$ is not so different to $\lf D\rf\cdot(\lf D\rf-K_Y)$ for large $D$. Suppose $k$ is large. From the observations following Prop.~\ref{prop:lim_asy} we let $D_k=\sqrt{\frac{2k}{A^2}}A$ be an approximate optimiser for the infimum defining $\casy_k(Y,A)$. Note that
$$\Delta_k=\sqrt{\frac{2k}{A^2}}A-\left\lf\sqrt{\frac{2k}{A^2}}A\right\rf$$
is a boundary (i.e.~its coefficients lie in $[0,1)$) for all $k$ and so $A\cdot\Delta_k$ is bounded. To compare $D_k\cdot(D_k-K_Y)$ and $\lf D_k\rf\cdot(\lf D_k\rf-K_Y)$ we compute
\begin{align*}
D_k\cdot(D_k-K_Y)-\lf D_k\rf\cdot(\lf D_k\rf-K_Y)&=2D_k\cdot\Delta_k+\Delta_k\cdot K_Y \\
&=2\sqrt{\frac{2k}{A^2}}A\cdot\Delta_k+\Delta_k\cdot K_Y
\end{align*}
It follows that $\calg_k(Y,A)\leq\casy_{k+O(\sqrt{k})}(Y,A)$ and hence
$$\limsup_{k\to\infty}\frac{\calg_k(Y,A)^2}{k}\leq\lim_{k\to\infty}\frac{\casy_{k+O(\sqrt{k})}(Y,A)^2}{k}=\lim_{k\to\infty}\frac{\casy_k(Y,A)^2}{k}$$
where the equality holds by Prop.~\ref{prop:lim_asy}. For the converse inequality, which will establish that the limit of algebraic capacities exists, we consider an optimal nef $\Z$-divisor for $\calg_k(Y,A)$. This is a nef divisor with $D\cdot (D-K_Y)\geq 2k$ that hence features in the infimum for $\casy_k(Y,A)$ and so
$$\casy_k(Y,A)\leq\calg_k(Y,A)$$
which completes the proof.
\end{proof}

\begin{prop} \label{prop:recur} Suppose $(Y,A)$ is a smooth pseudo-polarised surface. Assume that $A$ is a $\Z$-divisor, and let $D$ be an optimal $\Z$-divisor for $\calg_k(Y,A)$ with $k$ sufficiently large. Then $D+A$ is optimal for $\calg_{k'}(Y,A)$ where $k'=k+\frac{1}{2}I(A)+\calg_k(Y,A)$.
\end{prop}

\begin{proof} Write $D=aA+\sum b_ie_i$. As $D$ is optimal for $\calg_k(Y,A)$, we have $D\cdot A\leq D'\cdot A$ for all $\Z$-divisors $D'$ with $I(D')\geq 2k$. By the choice of $k'$, we have
$$\calg_{k'}(Y,A)\leq(D+A)\cdot A=\calg_k(Y,A)+A^2$$
Suppose $D'$ is a $\Z$-divisor with $I(D')\geq k'$ and $D'\cdot A<(D+A)\cdot A$. Write $D'=a'A+\sum b_i'e_i$ and consider
\begin{align*}
0&<(D+A)\cdot A-D'\cdot A \\
&=((a+1-a')A+\sum(b_i-b_i')e_i)\cdot A \\
&=(a+1-a')A^2
\end{align*}
and so $a'<a+1$. We claim that $D'-A$ is a strictly preferable divisor to $D$ and that, when $k$ is sufficiently large, $D'-A$ is effective. It follows then from Prop.~\ref{prop:ne_nef} that $\calg_k(Y,A)<D\cdot A$, which is a contradiction.

First, we note that the algebraic Weyl law for smooth surfaces gives that $\calg_k(Y,A)\to\infty$ as $k\to\infty$ and so, working in the given coordinates, we can choose $D'\cdot A$ sufficiently large. Observe then that $h^0(K_Y-D'+A)=0$ since
$$A\cdot(K_Y-D'+A)=A\cdot(K_Y+A)-D'\cdot A$$
is negative for large $k$ and so $K_Y-D'+A$ cannot be effective as $A$ is nef. Hence $h^2(D'-A)=0$ and we will subsequently show $I(D'-A)\geq I(D)\geq 2k$ giving, again for sufficiently large $k$, that we must have $h^0(D'-A)>0$.

We complete the proof by verifying that $D'-A$ would be a preferable candidate to $D$ for $\calg_k(Y,A)$ if such $D'$ existed. It is clear that $(D'-A)\cdot A<D\cdot A$ from the assumption $D'\cdot A<(D+A)\cdot A$. To compare indices, we compute
\begin{align*}
I(D'-A)&=(a'-1)(a'+c-1)A^2-\sum r_ib_i'(b_i'+d_i) \\
&=a'(a'+c)A^2-\sum r_ib_i'(b_i'+d_i)-(2a'+c-1)A^2 \\
&\geq 2k'-(2a'-2)A^2-(c+1)A^2 \\
&=2k'-2(D'-A)\cdot A-I(A) \\
&>2k'-2D\cdot A-I(A) \\
&=2k
\end{align*}
by the definition of $k'$.
\end{proof}

We say that $x\in\R$ is \emph{attained} by $(Y,A)$ if there exists some $k$ such that $\calg_k(Y,A)=x$. It follows from Prop.~\ref{prop:recur} there exists a set $S(Y,A)\subseteq[0,A^2)\cap\Z$ such that, for all sufficiently large $x\in\Z_{\geq0}$, $x$ is attained by $(Y,A)$ if and only if $x\equiv s\on{mod}{A^2}$ for some $s\in S(Y,A)$. We call $S(Y,A)$ the set of \emph{attained residues} of $(Y,A)$.

\begin{cor} \label{cor:fin_list} Suppose $(Y,A)$ is a pseudo-polarised smooth surface with $A$ a $\Z$-divisor. There exist $k_0,J\in\Z_{\geq0}$ such that
$$\{\calg_k(Y,A):k\geq k_0\}=\{s+jA^2:s\in S,j\geq J\}$$
\end{cor}

\begin{proof}
It suffices to choose $k_0$ large enough such that the recursion from Prop.~\ref{prop:recur} holds and such that $k\geq k_0$ implies $\calg_k(Y,A)\equiv s\on{mod}{A^2}$ for some $s\in S(Y,A)$, and to then choose a suitable $J$ using that $\calg_k(Y,A)$ is an increasing sequence.
\end{proof}

\begin{cor} \label{cor:bd_gap} Suppose $(Y,A)$ is a pseudo-polarised smooth surface with $A$ a $\Z$-divisor. Then
$$\limsup_{k\to\infty}\{\calg_{k+1}(Y,A)-\calg_k(Y,A)\}<\infty$$
\end{cor}

\begin{definition} \label{def:tc} Define the \emph{gap} of a weakly polarised surface $(Y,A)$ to be
$$\on{gap}(Y,A):=\limsup_{k\to\infty}\{\calg_{k+1}(Y,A)-\calg_k(Y,A)\}$$
We say that $(Y,A)$ is \emph{tightly-constrained} if $\on{gap}(Y,A)=1$.
\end{definition}

Note that $\on{gap}(Y,qA)=q\on{gap}(Y,A)$ for $q\in\R$ and so it follows that the gap is finite whenever $Y$ is smooth and $A$ is a real multiple of a $\Z$-divisor -- in particular, when $A$ is a $\Q$-divisor -- by Cor.~\ref{cor:bd_gap}. Note that, if $A$ is a $\Z$-divisor, $(Y,A)$ is tightly constrained if and only if there exists $r_0$ such that for all $r\in\Z_{\geq r_0}$ there exists $k$ with $\calg_k(Y,A)=r$. We call such $r_0$ a \emph{lower bound} for $(Y,A)$. Equivalently, $(Y,A)$ is tightly constrained if and only if $S(Y,A)=[0,A^2)\cap\Z$. In the smooth case this agrees with previous definitions of tightly constrained \cite[Def.~5.3]{bwo} and \cite[Def.~5.4.6]{bwt}. Lastly, observe that if we order the elements $s_1<s_2<\dots<s_m$ of $S(Y,A)$ then
$$\on{gap}(Y,A)=\on{max}\{s_\ell-s_{\ell-1}:\ell=2,\dots,m\}$$

\subsection{Sub-leading asymptotics for algebraic capacities}

Having the algebraic Weyl law for smooth surfaces suggests we make the following definition.

\begin{definition} We define the $k$th \emph{algebraic error term} of a weakly polarised $\Q$-factorial surface $(Y,A)$ to be
$$\ealg_k(Y,A):=\calg_k(Y,A)-\sqrt{2A^2k}$$
\end{definition}

When $(Y,A)$ is a pseudo-polarised smooth surface we see that $\ealg_k(Y,A)$ is $o(\sqrt{k})$. We will show much more.

\begin{thm} \label{thm:lims_z} Suppose $(Y,A)$ is a pseudo-polarised smooth surface with $A$ a $\Z$-divisor. Then
$$\limsup_{k\to\infty}\ealg_k(Y,A)=\on{gap}(Y,A)+\frac{1}{2}K_Y\cdot A\text{ and }\liminf_{k\to\infty}\ealg_k(Y,A)=\frac{1}{2}K_Y\cdot A$$
\end{thm}

We outline some consequences before proving Thm.~\ref{thm:lims_z}.

\begin{definition} For any weakly polarised $\Q$-factorial surface $(Y,A)$ define the \emph{algebraic Ruelle invariant} to be
$$\ralg(Y,A):=-K_Y\cdot A-\on{gap}(Y,A)$$
\end{definition}

By `$A$ is a real multiple of a $\Z$-divisor', we mean that there exists $q\in\R_{>0}$ and $A_0\in\on{WDiv}(Y)_\Z$ such that $A=qA_0$. The complementary case is when an $\R$-divisor $A$ is \emph{irrational}: namely, $A$ is irrational if $qA$ is not a $\Z$-divisor for all $q\in\R\setminus\{0\}$.

\begin{cor} \label{cor:lim_q} When $(Y,A)$ is a pseudo-polarised smooth surface with $A$ a real multiple of a $\Z$-divisor, the midpoint of
$$\limsup_{k\to\infty}\ealg_k(Y,A)\text{ and }\liminf_{k\to\infty}\ealg_k(Y,A)$$
is given by $-\frac{1}{2}\ralg(Y,A)$.
\end{cor}

\begin{proof} Let $q\in\Z_{\geq0}$ be such that $qA$ is a $\Z$-divisor. Then Thm.~\ref{thm:lims_z} gives that the midpoint of
$$\limsup_{k\to\infty}\ealg_k(Y,qA)\text{ and }\liminf_{k\to\infty}\ealg_k(Y,qA)$$
is $-\frac{1}{2}\ralg(Y,qA)$. It is clear that both $-K_Y\cdot qA$ and $\on{gap}(Y,qA)$ scale linearly with $q$ and hence $\ralg(Y,qA)$ also scales linearly with $q$. Thus the midpoint of
$$\limsup_{k\to\infty}\ealg_k(Y,A)\text{ and }\liminf_{k\to\infty}\ealg_k(Y,A)$$
is $-\frac{1}{2q}\ralg(Y,qA)=-\frac{1}{2}\ralg(Y,A)$.
\end{proof}

In order to prove Thm.~\ref{thm:lims_z} we will need to study the cap function more deeply. Recall that for smooth surfaces
$$\on{cap}_{(Y,A)}(x)=1+\frac{1}{2}\sup_{D\in\Nefqc(Y)}\{I(D):D\cdot A\leq x\}$$
Notice that if $D$ is preferable to $D'$ in the sense of Def.~\ref{def:pref} then it also preferable as a candidate for the optimisation problem defining the cap function.

\begin{lemma} \label{lem:cap_recur} Suppose $(Y,A)$ is a pseudo-polarised smooth surface with $A$ a $\Z$-divisor. Suppose $x$ is attained by $(Y,A)$ and let $D_0$ be an optimiser for $\on{cap}_{(Y,A)}(x)$. Then, when $x$ is sufficiently large, $D+A$ is an optimiser for $\on{cap}(x+A^2)$.
\end{lemma}

\begin{proof} Notice first that if $x$ is attained, say $\calg_k(Y,A)=x$, then
$$\on{cap}_{(Y,A)}(x)=1+\frac{1}{2}\sup_{D\in\on{nef}(Y)_\Q}\{I(D):D\cdot A=x\}$$
Suppose $D_0+A$ is not optimal. Then there is some $D'$ such that $I(D')>I(D_0+A)$ and $D'\leq x+A^2$. Observe that from Prop.~\ref{prop:recur} $x+A^2$ is attained by $(Y,A)$ for sufficiently large $x$, and so we may assume that $D'\cdot A=x+A^2$. Consider the divisor $D'-A$. By a similar argument as was used in the proof of Prop.~\ref{prop:recur} we can show that $D'-A$ is effective when $x$ is large enough. We now compute
\begin{align*}
I(D'-A)&=I(D')-2D'\cdot A+A^2+A\cdot K_Y \\
&>I(D_0+A)-2D'\cdot A+A^2+A\cdot K_Y \\
&=I(D_0)+2D_0\cdot A-2D'\cdot A+2A^2 \\
&=I(D_0)+2x-2(x+A^2)+2A^2 \\
&=I(D_0)
\end{align*}
Since $(D'-A)\cdot A=x$, and replacing the effective divisor $D'-A$ with a preferable nef divisor using Lem.~\ref{lem:iso_pref} if necessary, we see that $\on{cap}_{(Y,A)}(x)>I(D_0)$, which is a contradiction as we assumed that $D_0$ was optimal.
\end{proof}

\begin{cor} \label{cor:cap_recur} With assumptions as in Lem.~\ref{lem:cap_recur}, we have
$$\on{cap}_{(Y,A)}(x+A^2)=\on{cap}_{(Y,A)}(x)+x+\frac{1}{2}I(A)$$
whenever $x$ is sufficiently large and is attained by $(Y,A)$.
\end{cor}

\begin{proof}[Proof of Thm.~\ref{thm:lims_z}] Let $(Y,A)$ be a pseudo-polarised smooth surface with $A$ a $\Z$-divisor. By Cor.~\ref{cor:fin_list} there exist $k_0,J\in\Z_{\geq0}$ such that
$$\{\calg_k(Y,A):k\geq k_0\}=\{s+jA^2:j\in\Z_{\geq J},s\in S(Y,A)\}$$
Let $(x_i)$ be the sequence of distinct values of $\calg_k(Y,A)$ as $k\in\Z_{\geq0}$. Note that
$$\liminf_{k\to\infty}\ealg_k(Y,A)=\liminf_{i\to\infty}\{x_i-\sqrt{2A^2(\on{cap}_{(Y,A)}(x_i)-1)}\}$$
since the algebraic error term featuring $\calg_k(Y,A)=x_i$ is minimised when $k$ is as large as possible and this value of $k$ is exactly $\on{cap}_{(Y,A)}(x_i)-1$. Consider the subsequence of $(x_i)$ given by $y^s_j=s+jA^2$ where the index $j\geq J$. We see from Cor.~\ref{cor:cap_recur} that
$$\on{cap}_{(Y,A)}(y^s_{j+1})=\on{cap}_{(Y,A)}(y^s_j)+(s+jA^2)+\frac{1}{2}I(A)$$
Solving this difference equation in $j$ yields
$$\on{cap}_{(Y,A)}(y^s_j)=\frac{A^2}{2}j^2+(s-\frac{1}{2}K_Y\cdot A)j+\gamma_s$$
for some constant $\gamma_s\in\R$. This allows us to compute the limit
$$\lim_{j\to\infty}\{y^s_j-\sqrt{2A^2(\on{cap}_{(Y,A)}(y^s_j)-1)}\}=\lim_{j\to\infty}\{s+jA^2-\sqrt{(A^2)^2j^2+A^2(2s-K_Y\cdot A)j+\gamma_s''}\}$$
for some new constant $\gamma''_s$. Implicitising $j$ by solving $2A^2\on{cap}_{(Y,A)}(y^s_j)=k$ we can express this limit as
\begin{align*}
\lim_{k\to\infty}\{s+\frac{(K_Y\cdot A-2s)A^2}{2A^2}+\frac{\sqrt{(A^2)^2(K_Y\cdot A-2s)^2+4(A^2)^2(k-\gamma_s'')}}{2A^2}-\sqrt{k-1}\}=\frac{1}{2}K_Y\cdot A
\end{align*}
The sequence $(x_i)$ is hence a union of subsequences limiting to $\frac{1}{2}K_Y\cdot A$ and so
$$\liminf_{k\to\infty}\ealg_k(Y,A)=\lim_{i\to\infty}\{x_i-\sqrt{2A^2(\on{cap}_{(Y,A)}(x_i)-1)}\}=\frac{1}{2}K_Y\cdot A$$
We will show
$$\limsup_{k\to\infty}\ealg_k(Y,A)-\liminf_{k\to\infty}\ealg_k(Y,A)=\on{gap}(Y,A)$$
which gives the theorem. List the elements of $S(Y,A)$ in increasing order: $s_1<s_2<\dots<s_m$. We have
$$\limsup_{k\to\infty}\ealg_k(Y,A)=\limsup_{k\to\infty}\{x_i-\sqrt{2A^2\on{cap}_{(Y,A)}(x_i-1)}\}$$
since $\on{cap}_{(Y,A)}(x_i-1)$ is the smallest $k$ with $\calg_k(Y,A)=x_i$. We again consider the subsequences $(y^s_j)$. For each of these subsequences we can compute
\begin{align*}
\lim_{j\to\infty}\{y^{s_\ell}_j-&\,\sqrt{2A^2\on{cap}_{(Y,A)}(y^{s_\ell}_j-1)}\} \\
&=\lim_{j\to\infty}\{s_\ell+jA^2-\sqrt{2A^2\on{cap}_{(Y,A)}(s_\ell+jA^2-1)}\} \\
&=\lim_{j\to\infty}\{s_\ell+jA^2-\sqrt{2A^2\on{cap}_{(Y,A)}(s_{\ell-1}+jA^2)}\} \\
&=(s_\ell-s_{\ell-1})+\lim_{j\to\infty}\{s_{\ell-1}+jA^2-\sqrt{2A^2\on{cap}_{(Y,A)}(s_{\ell-1}+jA^2)}\} \\
&=(s_\ell-s_{\ell-1})+\lim_{j\to\infty}\{s_{\ell-1}+jA^2-\sqrt{2A^2(\on{cap}_{(Y,A)}(s_{\ell-1}+jA^2)-1)}\} \\
&=(s_\ell-s_{\ell-1})+\liminf_{k\to\infty}\ealg_k(Y,A)
\end{align*}
Again, since the $(y^{s_\ell}_j)$ cover the whole sequence $(x_i)$, we see that
\begin{align*}
\limsup_{k\to\infty}\ealg_k(Y,A)&=\underset{2\leq\ell\leq m}{\on{max}}(s_\ell-s_{\ell-1})+\liminf_{k\to\infty}\ealg_k(Y,A) \\
&=\on{gap}(Y,A)+\liminf_{k\to\infty}\ealg_k(Y,A)
\end{align*}
as required.
\end{proof}

We deduce the following result, which will also be of use in other contexts.

\begin{lemma} \label{lem:fin_div} Suppose $(Y,A)$ is a pseudo-polarised smooth surface with $A$ a $\Z$-divisor. There exists a finite list of nef $\Z$-divisors $D_1,\dots,D_n$ on $Y$ such that, for all sufficiently large $k$, optimisers for $\calg_k(Y,A)$ can be chosen to take the form
$$D_i+jA$$
for some $i\in\{1,\dots,n\}$ and some $j\in\Z_{\geq0}$.
\end{lemma}

\begin{proof} Suppose $x$ is attained by $(Y,A)$. Recall that a divisor $D$ that is optimal for $\on{cap}_{(Y,A)}(x)$ is also optimal for $\calg_{k(x)}(Y,A)$ where $k(x)$ is the largest integer $k$ such that $\calg_k(Y,A)\leq x$ (or such that $\calg_k(Y,A)=x$). Note that such $D$ is then optimal for the optimisation problems for all $\calg_k(Y,A)$ such that $\calg_k(Y,A)=x$. Let $x_0\in\Z_{\geq0}$ be such that the recursion from Lem.~\ref{lem:cap_recur} holds for all $x\geq x_0$. Let $k_0\in\Z_{\geq0}$ such that $\calg_k(Y,A)\geq x_0$ for all $k\geq k_0$. Without loss of generality choose $k_0$ such that
$$\{\calg_k(Y,A):k\geq k_0\}=\{s+jA^2:s\in S(Y,A),j\geq J\}$$
for some $J\in\Z_{\geq0}$ as in Cor.~\ref{cor:fin_list}. Choose $k_i\geq k_0$ such that $\calg_{k_i}(Y,A)\equiv s_i\on{mod}{A^2}$ for each $s_i\in S(Y,A)$. Set $x_i=\calg_{k_i}(Y,A)$ and let $D_i$ be optimal for $\on{cap}_{(Y,A)}(x_i)$. Then we see from Lem.~\ref{lem:cap_recur} that $D_i+jA$ is optimal for $\on{cap}_{(Y,A)}(x_i+jA^2)$ and in particular that $D_i+jA$ is optimal for all the optimisation problems for $\calg_k(Y,A)$ such that $\calg_k(Y,A)=x_i+jA^2$. Since this covers all values of $\calg_k(Y,A)$ for $k$ large by Cor.~\ref{cor:fin_list} we see that optimisers for $\calg_k(Y,A)$ for all sufficiently large $k$ can be chosen to take the form $D_i+jA$ where $i$ is such that $\calg_k(Y,A)\equiv s_i\on{mod}{A^2}$.
\end{proof}

\subsection{Tightly-constrained polarised surfaces}

We will show the following theorem, which will help establish \cite[Conj.~5.7]{bwo} by characterising when $(Y,A)$ is tightly constrained.

\begin{thm} \label{thm:gap_inf} Let $(Y,A)$ be a pseudo-polarised smooth surface with $A$ a $\Z$-divisor. Then
$$\on{gap}(Y,A)=\inf_{D\in N^1(Y)_\Z}\{D\cdot A:D\cdot A>0\}$$
\end{thm}

Observe that since the map $N^1(Y)_\Z\to\Z,D\mapsto D\cdot A$ is a group homomorphism the quantity $\inf_{D\in N^1(Y)_\Z}\{D\cdot A:D\cdot A>0\}$ is the positive generator of the image of this homomorphism.

\begin{proof} Let $p=\inf_{D\in N^1(Y)_\Z}\{D\cdot A:D\cdot A>0\}$. It is clear that $\on{gap}(Y,A)\geq p$. For the converse inequality,
suppose that $D_*$ is a $\Z$-divisor witnessing the infimum; so $D_*\cdot A=p$. Consider the optimisation problem
$$\on{cap}_{(Y,A)}(x)=1+\frac{1}{2}\on{max}\{I(D):D\cdot A\leq x\}$$
We claim that for any $x\in\Z_{\geq0}$ such that there exists $k\in\Z_{\geq0}$ with $\calg_k(Y,A)=x$ we have that an optimiser $D$ for $\on{cap}_{(Y,A)}(x)$ must satisfy $D\cdot A=x$. Suppose not; let $D$ be an optimiser with $D\cdot A<x$. By definition $p$ divides $x-D\cdot A$. Let $x-D\cdot A=pm$ and consider the $\Z$-divisor $D'=D+mD_*$. We compute
$$D'\cdot A=x\text{ and }I(D')=I(D)+2mD\cdot D_*+I(mD_*)$$
Notice that a nef $\Z$-divisor that is optimal for $\on{cap}_{(Y,A)}(x)$ for some $x\in\Z_{\geq0}$ is also optimal for $\calg_{k(x)}(Y,A)$ where $k(x)$ is the largest index $k$ such that $\calg_k(Y,A)\leq x$. Lem.~\ref{lem:fin_div} implies that $D=D_i+jA$ where $D_i$ comes from a finite list of $\Z$-divisors that is independent of $x$. When $x$ is sufficiently large -- and hence $j$ is also large -- it is easy to see that $D'$ is effective. We compute
$$D\cdot D_*=D_i\cdot D_*+j$$
which is positive when $x$ is sufficiently large. Therefore,
$$D'\cdot A=x\text{ and }I(D')>I(D)$$
and so we see that $D$ is not optimal for $\on{cap}_{(Y,A)}(x)$. Hence, when $x\in\Z_{\geq0}$ is such that $\calg_k(Y,A)=x$ for some $k\in\Z_{\geq0}$, optimisers $D$ for $\on{cap}_{(Y,A)}(x)$ must have $D\cdot A=x$. For $x=\calg_k(Y,A)$, observe that there exists $k$ with $\calg_k(Y,A)=x+p$ when $x$ is sufficiently large. This is true since, if $D$ optimises $\calg_k(Y,A)$, then $(D+D_*)\cdot A=\calg_k(Y,A)+p$ and, $I(D+D_*)>2k$ from the previous index calculation when $x$ is sufficiently large. It follows that for all $x\in\Z_{\geq0}$
$$\on{cap}_{(Y,A)}(x+p)>\on{cap}_{(Y,A)}(x)$$
and so $\on{gap}(Y,A)\leq p$.
\end{proof}

\begin{cor} \label{cor:alg_tc} Let $(Y,A)$ be a pseudo-polarised smooth surface with $A$ a real multiple of a $\Z$-divisor. Then $(Y,A)$ is tightly constrained if and only if
$$\inf_{D\in N^1(Y)_\Z}\{D\cdot A:D\cdot A>0\}=1$$
\end{cor}

\begin{proof} The result follows immediately from Thm.~\ref{thm:gap_inf} when $A$ is a $\Z$-divisor. Suppose $A=qA_0$ as above. As discussed previously, $\on{gap}(Y,A)=q\cdot\on{gap}(Y,A_0)$ and
$$\on{gap}(Y,A_0)=\inf_{D\in N^1(Y)_\Z}\{D\cdot A_0:D\cdot A>0\}$$
by Thm.~\ref{thm:gap_inf} since $A_0$ is a $\Z$-divisor. Notice that the infimum also scales linearly with $A$ and so
$$\on{gap}(Y,A)=q\cdot\on{gap}(Y,A_0)=q\cdot\inf_{D\in N^1(Y)_\Z}\{D\cdot A:D\cdot A_0>0\}=\inf_{D\in N^1(Y)_\Z}\{D\cdot A:D\cdot A>0\}$$
which gives the result.
\end{proof}

\begin{cor} Suppose that $(Y,A)$ is a pseudo-polarised smooth surface with $A$ a real multiple of a $\Z$-divisor. For sufficiently large $k$ we have $\calg_{k+1}(Y,A)-\calg_k(Y,A)\in\{0,\on{gap}(Y,A)\}$.
\end{cor}

\begin{proof} First write $A=qA_0$ for some nef $\Z$-divisor $A_0$ and let $p=\on{gap}(Y,A_0)$. Then, since $\calg_k(Y,A_0)\in\Z$, by definition of the gap there exists $k_0$ such that $k\geq k_0$ implies $0\leq\calg_{k+1}(Y,A)-\calg_k(Y,A)\leq p$. By the infimum expression of $\on{gap}(Y,A_0)$ in Thm.~\ref{thm:gap_inf} we must then have $\calg_{k+1}(Y,A)-\calg_k(Y,A)\in\{0,p\}$ for all $k\geq k_0$. The result for $A$ then follows by linearity in $q$.
\end{proof}

\subsection{Asymptotics for singular toric surfaces}

We consider pseudo-polarised singular surfaces. The continuity results of \S\ref{sec:degen} allow us to extend the algebraic Weyl law to singular toric surfaces.

\begin{prop} \label{prop:sing_asy} Suppose $(Y,A)$ is a pseudo-polarised toric surface. Then
$$\lim_{k\to\infty}\frac{\calg_k(Y,A)^2}{k}=2A^2$$
If $A$ is not irrational then
$$\limsup_{k\to\infty}\ealg_k(Y,A)=\on{gap}(Y,A)+\frac{1}{2}K_Y\cdot A\text{ and }\liminf_{k\to\infty}\ealg_k(Y,A)=\frac{1}{2}K_Y\cdot A$$
\end{prop}

\begin{proof} Let $\pi\colon\wt{Y}\to Y$ be a toric resolution. By Prop.~\ref{prop:degen_toric} we have
$$\calg_k(\wt{Y},\pi^*A)=\calg_k(Y,A)$$
and the result follows from Thm.~\ref{thm:smooth_asy} and Thm.~\ref{thm:lims_z}.
\end{proof}

It would be interesting to identify if all pseudo-polarised surfaces $(Y,A)$ with $A$ irrational have the property that $\ealg_k(Y,A)$ is convergent for applications to symplectic geometry \cite[Conj.~1.5]{hu}. Following \cite[Rmk.~1.15]{hu} we also conjecture that $\ealg_k(Y,A)\leq 0$ at least for all pseudo-polarised toric surfaces.

Lastly, we distill when $(Y,A)$ is tightly-constrained in the toric setting.

\begin{prop} \label{lem:toric_tc} Let $(Y,A)$ be a pseudo-polarised smooth toric surface. Then $(Y,A)$ is tightly constrained if and only if $A$ is a primitive $\Z$-divisor.
\end{prop}

\begin{proof} Let $p=\inf_{D\in N^1(Y)_\Z}\{D\cdot A:D\cdot A>0\}$, and let $\Omega=P(A)$. Notice that $A$ is primitive and Cartier if and only if all its edge lengths are integral and
$$\on{gcd}\{\ell_{\Z^2}(e):e\in\on{Edge}(\Omega)\}=1$$
In this case there exist integers $\{a_e\}_{e\in\on{Edge}(\Omega)}$ such that $\sum_{e\in\on{Edge}(\Omega)}a_e\ell_{\Z^2}(e)=1$. It follows that the divisor $D=\sum_{e\in\on{Edge}(\Omega)}a_eD_e$ has $D\cdot A=1$, where $D_e$ is the prime divisor on $Y$ corresponding to the edge $e$. Thm.~\ref{thm:gap_inf} implies that $\on{gap}(Y,A)=1$ and so $(Y,A)$ is tightly constrained. The converse is clear: if $A=qA_0$ for some $q\in\R_{>0}$ and for some primitive $\Z$-divisor $A_0$, then $\calg_k(Y,A)=q\cdot\calg_k(Y,A_0)$ and so $\on{gap}(Y,A)=q$ from the previous calculation.
\end{proof}

\subsection{Presentations of the cap function}

We discuss how to generalise \cite[Thm.~1.1]{bwo} to provide presentations of the cap function of a polarised surface $(Y,A)$ as a kind of Hilbert function or, when $Y$ is toric, as an Ehrhart function. In situations where there is a good Riemann--Roch formula available -- such as when $Y$ has orbifold singularities and in particular when $Y$ is toric -- this provides an explicit eventual expression for $\on{cap}_{(Y,A)}(x)$, often as a quasi-polynomial.

\begin{prop} \label{prop:cap_rep} Suppose that $(Y,A)$ is a pseudo-polarised surface with $Y$ either smooth or toric. Suppose that $A$ is a Cartier divisor on $Y$. Then there exist $\gamma_0,\dots,\gamma_{A^2-1}\in\Q$ such that, for sufficiently large $x\in\Z_{\geq0}$,
\begin{equation} \on{cap}_{(Y,A)}(x)=\frac{1}{2A^2}x^2+\frac{-K_Y\cdot A}{2A^2}x+\gamma_i \tag{$\spadesuit$} \label{eqn:qp_alg}
\end{equation}
where $x\equiv i\on{mod}{A^2}$. Equivalently,
$$\on{cap}_{(Y,A)}(i+jA^2)=\chi(D_i+jA)-\chi(\mO_Y)+1$$
for each $i=0,\dots,A^2-1$ and for sufficiently large $j$, where $D_i$ is from a finite list of divisors as in Lem.~\ref{lem:fin_div}. When $Y$ is toric, this takes the form
$$\on{cap}_{(Y,A)}(i+jA^2)=\#\{P_i+jP(A)\}\cap\Z^2$$
for some finite list of polygons $P_0,\dots,P_{A^2-1}$.
\end{prop}

This expresses the cap function of $(Y,A)$ as a `multigraded Hilbert function' (resp.~`mixed Ehrhart function' \cite{hjst}) that counts sections (resp.~lattice points) for linear combinations of divisors (resp.~polytopes).

\begin{proof} The explicit formula comes from the recursion from Cor.~\ref{cor:cap_recur}. The other formulations follow from interpreting $\on{cap}_{(Y,A)}(x)-1$ as the largest $k$ such that $\calg_k(Y,A)\leq x$, which is given respectively by the Euler characteristics or the lattice point counts.
\end{proof}

We can use Prop.~\ref{prop:cap_rep} to compute $\on{cap}_{(Y,A)}(x)$ explicitly in cases when $A$ is a $\Q$-divisor or a $\Q$-Cartier $\Z$-divisor, and non-explicitly when $A$ is a real multiple of a $\Z$-divisor as in \cite[Ex.~5.11]{bwo}. We can also derive a bound for when the quasi-polynomial behaviour begins.

\begin{prop} \label{prop:low_bound} Suppose $(Y,A)$ is a tightly-constrained pseudo-polarised smooth surface. Then the cap function $\on{cap}_{(Y,A)}(x)$ is given by the quasi-polynomial (\ref{eqn:qp_alg}) for all $x\geq x_0$ where $x_0$ satisfies
\begin{enumerate}
\item $x_0>A\cdot(A+K_Y)$, \label{item:1}
\item all values $x_0,x_0+1,\dots,x_0+A^2-1$ are attained by $(Y,A)$. \label{item:2}
\end{enumerate}
If $(Y,A)$ is a pseudo-polarised (possibly singular) toric surface corresponding to a primitive lattice polytope $\Omega$ then $\on{cap}_{(Y,A)}(x)$ is given by a quasi-polynomial for all $x\geq x_0$ where $x_0$ satisfies
\begin{enumerate}
\item[\textnormal{(a)}] $x_0>2\on{vol}(\Omega)-\ell_{\Z^2}(\partial\Omega)$
\item[\textnormal{(b)}] all values $x_0,x_0+1,\dots,x_0+2\on{vol}(\Omega)-1$ are attained by $(Y,A)$. \label{item:b}
\end{enumerate}
\end{prop}

Recall that (\ref{item:2}) means that there exist $k_0,\dots,k_{A^2-1}\in\Z_{\geq0}$ such that
$$\calg_{k_i}(Y,A)=x_0+i$$
or, equivalently, that
$$\on{cap}_{(Y,A)}(x_0+i-1)<\on{cap}_{(Y,A)}(x_0+i)$$
for $i=0,\dots,A^2-1$.

\begin{proof}
Note that the claim for smooth toric surfaces follows immediately from interpreting $A^2$ and $A\cdot K_Y$ combinatorially, from which the singular case follows by Prop.~\ref{prop:degen_toric}.

Analysing the proof of Lem.~\ref{lem:cap_recur} we observe that if $D$ is optimal for $\on{cap}_{(Y,A)}(x)$ and $x$ is attained then $D+A$ is optimal for $\on{cap}_{(Y,A)}(x+A^2)$ when $x$ is large enough for the following two properties to hold:
\begin{itemize}
\item $x+A^2$ is attained,
\item $D'\cdot A>A\cdot (A+K_Y)$ when $D'$ is as in the proof of Lem.~\ref{lem:cap_recur}.
\end{itemize}
Let $D$ be an optimal nef $\Z$-divisor for $\calg_k(Y,A)=x$. We observe from Lem.~\ref{prop:recur} that for $x+A^2$ to be attained it suffices to check that
$$C\cdot A>A\cdot (A+K_Y)\text{ whenever }\chi(C)\geq k+\frac{1}{2}I(A)+x=:k'$$
Note that $C\cdot A\geq\calg_{k'}(Y,A)\geq\calg_k(Y,A)=x$ and so the first property is achieved whenever we have $x>A\cdot(A+K_Y)$.

Now suppose $D$ is optimal for $\on{cap}_{(Y,A)}(x)$ where $x$ is attained. It follows that $D\cdot A=x$. It suffices to identify which $x$ have that $D'\cdot A>A\cdot(A+K_Y)$ whenever $I(D')>I(D+A)$. One can verify directly that $I(D+A)>I(D)$ when $x>A\cdot(A+K_Y)$ and so it follows that for such $x$ we have $D'\cdot A>x$ since $D$ is optimal for $\on{cap}_{(Y,A)}(x)$. The second property is thus achieved whenever $x>A\cdot(A+K_Y)$.

We have shown that the cap function obeys a recursion with period $A^2$ for all attained $x$ satisfying (\ref{item:1}). The condition (\ref{item:2}) ensures that all residues mod $A^2$ are subject to this recursion when $x\geq x_0$, and hence that it expresses $\on{cap}_{(Y,A)}(x)$ for $x\geq x_0$ as a quasi-polynomial.
\end{proof}

\section{ECH and algebraic capacities} \label{sec:ech}

Embedded Contact Homology (ECH) gives strong obstructions to the existence of symplectic embeddings. We will briefly summarise the theory -- largely from \cite{mh2} -- and deduce results in ECH that follow from our understanding of algebraic capacities.

\subsection{Review of ECH}

ECH is a symplectic model for Seiberg--Witten Floer homology, assigning a chain complex $\text{ECC}(Y,\lambda,J,\Gamma)$ of $\Z/2$-vector spaces to the data:
\begin{itemize}
\item a contact $3$-manifold $(Y,\lambda)$ with $\lambda$ generic,
\item a generic almost-complex structure $J$ on $Y\times\R$,
\item $\Gamma\in H_1(Y)$
\end{itemize}
The homology of this chain complex is the \emph{Embedded Contact Homology} $\text{ECH}(Y,\lambda,\Gamma)$, which is independent of $J$ as suggested by the notation. There is a degree $-2$ map $U$ on $\on{ECH}(Y,\lambda,0)$ and a filtration by the symplectic action. The differential decreases action and so there is a well-defined filtration $\text{ECH}^L(Y,\lambda,0)$ on homology.
One defines
$$c_k(Y,\lambda):=\inf\{L:\exists\eta\in\text{ECH}^L(Y,\lambda,0)\text{ with $U^k\eta = [\emptyset]$}\}$$
If $(X,\omega)$ is a symplectic filling of $(Y,\lambda)$ then we define the $k$th \emph{ECH capacity} of $(X,\omega)$ by
$$\cech_k(X,\omega):=c_k(Y,\lambda)$$
which one can show to be independent of the choice of $\lambda$. The ECH capacities satisfy several properties similar to those from \S\ref{sec:properties}:
\begin{itemize}
\item (Monotonicity) If $(X,\omega)$ symplectically embeds into $(X',\omega')$ then
$$\cech_k(X,\omega)\leq\cech_k(X',\omega')$$
\item (Conformality) For each $q\in\R_{>0}$,
$$\cech_k(X,q\omega)=q\cech_k(X,\omega)$$
\item (Disjoint Union) We have
$$\cech_k(\amalg(X_i,\omega_i))=\sup_{\sum k_i=k}\sum\cech_{k_i}(X_i,\omega_i)$$
\end{itemize}
As in the case of algebraic capacities we define the \emph{cap function} of $(X,\omega)$ by
$$\on{cap}_{(X,\omega)}(x):=\#\{k:\cech_k(X,\omega)\leq x\}$$

\subsection{Toric domains} \label{sec:tor_ech}

A class of spaces in which ECH is especially computable is toric domains. Consider the moment map $\mu\colon\C^2\to\R^2$ for the standard compact $2$-torus action on $\C^2$. Fix a region $\Omega\subseteq\R^2$ and define
$$X_\Omega:=\mu^{-1}(\Omega)$$
to be the \emph{toric domain} associated to $\Omega$. Cases of particular interest arise from putting additional constraints on $\Omega$.

\begin{definition} \label{def:conv_dom} We call $\Omega\subseteq\R^2$...
\begin{itemize}
\item a \emph{convex domain} if $\Omega\subseteq\R_{\geq0}^2$ and the part of boundary away from the axes is the graph of a concave non-increasing function.
\item a \emph{lattice (resp.~rational) convex domain} if $\Omega$ is a convex domain given by the graph of a piecewise-linear function such that the vertices of $\Omega$ lie in $\Z^2$ (resp.~$\Q^2$).
\item a \emph{rational-sloped convex domain} if $\Omega$ is a convex domain defined by the graph of a piecewise-linear function such that the slopes of the boundary of $\Omega$ are rational.
\item a \emph{concave domain} if $\Omega\subseteq\R_{\geq0}^2$ and the part of boundary away from the axes is the graph of a convex function.
\end{itemize}
\end{definition}

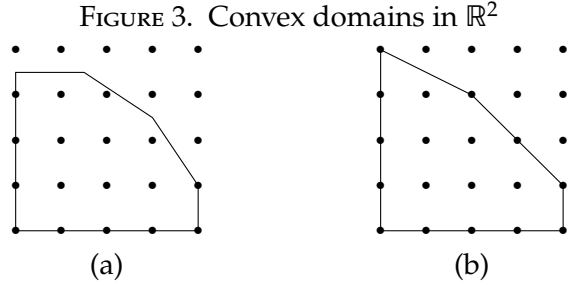
\begin{figure}[h]
\caption{Convex domains in $\R^2$}
\label{fig:conv_doms}
\begin{center}
\begin{tikzpicture}[scale=0.6]
\foreach \i in {0,...,4}
{
\foreach \j in {0,...,4}
{\node (\i\j) at (\i,\j){\tiny $\bullet$};
\node (\i\j) at (8+\i,\j){\tiny $\bullet$};
}
}

\draw (0,0) to (4,0);
\draw (0,0) to (0,3.5) to (1.5,3.5) to (3,2.5) to (4,1) to (4,0);

\draw (8,0) to (12,0);
\draw (8,0) to (8,4) to (10,3) to (12,1) to (12,0);

\node (l1) at (2,-0.8){(a)};
\node (l1) at (10,-0.8){(b)};
\end{tikzpicture}
\end{center}
\end{figure}

We show a rational convex domain in Fig.~\ref{fig:conv_doms}(a) and a lattice convex domain in Fig.~\ref{fig:conv_doms}(b). Note that lattice, rational, or rational-sloped convex domains are convex polygons. When $\Omega$ has a property $\mathcal{P}$ we say that $X_\Omega$ is a `$\mathcal{P}$ toric domain'. For instance, when $\Omega$ is a lattice convex domain we say that $X_\Omega$ is a lattice convex toric domain. Choi-Cristofaro-Gardiner--Frenkel--Hutchings--Ramos and Choi--Cristofaro-Gardiner produce a combinatorial formulation to compute ECH capacities for concave and convex toric domains. For a convex or concave domain $\Omega\subseteq\R^2$ and a vector $v\in\R^2$ define the $\Omega$-length
$$\ell_\Omega(v):=v\times p_v$$
where $p_v\in\partial\Omega$ is a point at which $v$ is tangent to $\partial\Omega$. If $\Lambda$ is a convex polygon, we define the $\Omega$-perimeter of $\Lambda$, $\ell_\Omega(\partial\Lambda)$, to be the sum of the $\Omega$-lengths of the boundary edges of $\Lambda$. When $\Lambda$ is a polygon we define $L_\Lambda=\#\Lambda\cap\Z^2$. When $\Lambda$ is a convex domain we define $\wh{L}_\Lambda$ to be the number of lattice points in the smallest rectangle containing $\Lambda$ minus the number of lattice points contained in $\Lambda$.

\begin{thm}[{\cite[Thm.~1.21]{ccfhr} + \cite[Cor.~A.12]{cg}}] Suppose $\Omega$ is a concave toric domain. Then
$$\cech_k(X_\Omega)=\on{max}\{\ell_\Omega(\partial\Lambda):\wh{L}_\Lambda=k+1\}$$
Suppose $\Omega$ is a convex toric domain. Then
$$\cech_k(X_\Omega)=\on{min}\{\ell_\Omega(\partial\Lambda):L_\Lambda=k+1\}$$
Both extrema range over lattice convex domains $\Lambda\subseteq\R^2$.
\end{thm}

Algebraic capacities are a substantial generalisation of this sort of optimisation problem designed to compute ECH capacities with objects coming from a different context.

Let $\Delta_a$ denote the triangle with vertices $(0,0),(a,0),(0,a)$.

\begin{definition} \label{def:conc_wt} Let $\Omega$ be a concave domain. The \emph{weight sequence} $w(\Omega)$ of $\Omega$ is defined recursively as follows.
\begin{itemize}
\item Set $w(\emptyset)=\emptyset$ and $w(\Delta_a)=(a)$.
\item Otherwise let $a$ be the largest real number such that $\Delta_a\subseteq\Omega$. This divides $\Omega$ into three (possibly empty) pieces: $\Delta_a,\Omega_1,\Omega_2$.
\item If not empty, $\Omega_1$ and $\Omega_2$ are affine-equivalent to concave domains. Define $w(\Omega)=(a,w(\Omega_1),w(\Omega_2))$.
\end{itemize}
\end{definition}

Note that $w(\Omega)$ is finite if $\Omega$ is a real multiple of a lattice concave domain but will be infinite in general. We define an analogous sequence for convex domains.

\begin{definition} Let $\Omega$ be a convex domain. The \emph{weight sequence} $w(\Omega)$ of $\Omega$ is defined recursively as follows.
\begin{itemize}
\item Let $a$ be the smallest real number such that $\Omega\subseteq B(a)$.
\item This divides $\Delta_a$ into three (possibly empty) pieces: $\Omega,\Omega_1,\Omega_2$.
\item If non-empty, $\Omega_1$ and $\Omega_2$ are affine-equivalent to concave domains. Define $w(\Omega)=(a;w(\Omega_1),w(\Omega_2))$ using Def.~\ref{def:conc_wt}.
\end{itemize}
\end{definition}

We note that the first element of the weight sequence for convex domains is distinguished. We consider all other terms of a weight sequence to be unordered and counted with repetition. We depict the decompositions used to recursively define the weight sequence in Fig.~\ref{fig:wt_seq}, with the concave case shown in Fig.~\ref{fig:wt_seq}(a) and the convex case in Fig.~\ref{fig:wt_seq}(b). In both cases we denote parts of the boundary of the triangle $\Delta_a$ by dashed lines.

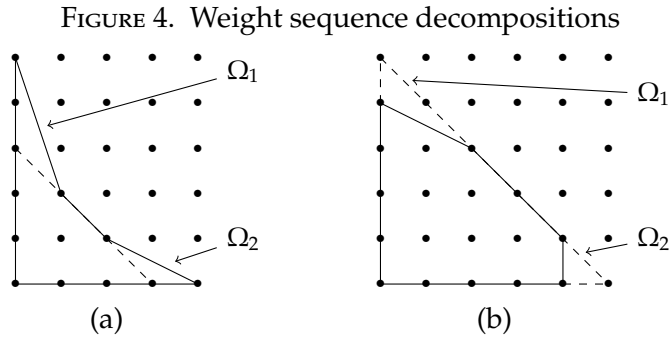
\begin{figure}[h]
\caption{Weight sequence decompositions}
\label{fig:wt_seq}
\begin{center}
\begin{tikzpicture}[scale=0.6]
\foreach \i in {0,...,4}
{
\foreach \j in {0,...,5}
{\node (\i\j) at (\i,\j){\tiny $\bullet$};
\node (\i\j) at (8+\i,\j){\tiny $\bullet$};
}
\node (\i) at (13,\i){\tiny $\bullet$};
}

\node (5) at (13,5){\tiny $\bullet$};

\draw (0,0) to (4,0);
\draw (0,0) to (0,5) to (1,2) to (2,1) to (4,0);
\draw[dashed] (0,3) to (3,0);

\node (l2) at (5,1){\small $\Omega_2$};
\node (l1) at (5,4.7){\small $\Omega_1$};

\draw[->] (l2) to (3.5,0.5);
\draw[->] (l1) to (0.7,3.5);

\draw (8,0) to (12,0);
\draw (8,0) to (8,4) to (10,3) to (12,1) to (12,0);
\draw[dashed] (8,5) to (13,0);
\draw[dashed] (8,4) to (8,5);
\draw[dashed] (12,0) to (13,0);

\node (l2) at (14,1){\small $\Omega_2$};
\node (l1) at (14,4.2){\small $\Omega_1$};

\draw[->] (l2) to (12.5,0.7);
\draw[->] (l1) to (8.7,4.5);

\node (la) at (2,-0.8){(a)};
\node (lb) at (10.5,-0.8){(b)};
\end{tikzpicture}
\end{center}
\end{figure}

\subsection{Relating ECH capacities and algebraic capacities}

The work of \cite{bwo,bwt,cw} establish strong connections between ECH capacities and algebraic capacities. First, the algebraic capacities of $(Y,A)$ create obstructions to symplectic embeddings into $(Y,\omega_A)$ where $\omega_A$ is the Poincar\'e dual of $A$.

\begin{thm}[{\cite[Thm.~1.3]{cw}}] \label{thm:cw} Suppose $(Y,A)$ is a polarised smooth rational surface. If $(X,\omega)$ is a star-shaped domain in $\R^4$ that symplectically embeds into $(Y,\omega_A)$ then
$$\cech_k(X,\omega)\leq\calg_k(Y,A)$$
\end{thm}

In particular, the algebraic capacities of $Y$ are related to the ECH capacities of the complement $X=Y\setminus\on{supp}(A)$ equipped with the restriction of $\omega_A$. It is conjectured that equality often holds when $Y$ is rational \cite[Conj.~1.7]{cw}. One situation in which algebraic capacities have been seen to equal ECH capacities is for rational-sloped convex toric domains.

\begin{thm}[{\cite[Thm.~1.5]{bwo}}] \label{thm:bwo} Suppose $\Omega$ is a rational-sloped convex domain. Then
$$\cech_k(X_\Omega)=\calg_k(Y_\Omega,A_\Omega)$$
where $Y_\Omega$ is the toric surface corresponding to the inner normal fan of $\Omega$ and $A_\Omega$ is the ample $\R$-divisor on $Y_\Omega$ whose polytope is $\Omega$.
\end{thm}

We see that $\Omega$ is required to have rational slopes so that its inner normal fan is rational and hence defines a toric surface. We note that identifying algebro-geometric versions of objects from Seiberg--Witten theory -- such as \cite{poin1,poin2} for Seiberg--Witten invariants -- has often been seen to be a fruitful line of inquiry. ECH makes substantial contact with Seiberg--Witten theory and so the framework of algebraic capacities can be viewed in connection to this wider story.

Thm.~\ref{thm:cw} and Thm.~\ref{thm:bwo} allows us to transport the results of this paper for algebraic capacities into ECH.

\subsection{Tightly-constrained convex domains}

We recall one of the key notions from \cite{bwo} and apply the technology of algebraic capacities developed in \S\ref{sec:asymptotics} to prove a conjecture from \emph{ibid}.

\begin{definition} We say that a convex domain $\Omega$ is \emph{tightly-constrained} if there exists $r_0$ such that for all $r\geq r_0$ there is $k$ with $\cech_k(X_\Omega)=r$.
\end{definition}

The value of this definition is that it defines a good setting to derive quasi-polynomial representations for the cap function in ECH; c.f.~\cite[Thm.~1.1]{bwo}. From Thm.~\ref{thm:bwo} we see that $\Omega$ is tightly-constrained if and only if $(Y_\Omega,A_\Omega)$ is tightly-constrained according to Def.~\ref{def:tc}.

We define $\on{gcd}{S}$ for some subset $S\subseteq\R_{>0}$ as the largest $t\in\R_{>0}$ such that for each $s\in S$ there exists $n\in\Z$ with $s=tn$. If no such $t$ exists, we define $\on{gcd}{S}=\infty$. This recovers the usual gcd for subsets of $\Z$. The following conjecture was stated in \cite{bwo}.

\begin{conjecture}[{\cite[Conj.~5.7]{bwo}}] \label{conj:tc} A convex lattice domain $\Omega$ is tightly-constrained if and only if $\on{gcd}w(\Omega)=1$.
\end{conjecture}

We denote by $\on{Edge}(\Omega)$ the set of edges of a polygon $\Omega$ and we denote the lattice length of an edge $e\in\on{Edge}(\Omega)$ by $\ell_{\Z^2}(e)$. We reduce the conjecture to a statement that does not involve weight sequences, and that can hence be translated in terms of algebraic capacities.

\begin{lemma} \label{lem:wt_seq} A lattice convex domain $\Omega$ is primitive if and only $\on{gcd}{w(\Omega)}=1$.
\end{lemma}

\begin{proof} It is clear that $\on{gcd}{w(\Omega)}>1$ implies that $\Omega$ is not primitive. We induct on the length of the weight sequence $w(\Omega)$ to prove the converse. If $w(\Omega)$ has length $1$ the result holds, since $\on{gcd}{w(\Omega)}=1$ if and only if $w(\Omega)=1$ in which case $\Omega$ is a primitive lattice triangle.

Suppose $w(\Omega)=(c;a_1,\dots,a_r;b_1,\dots,b_s)$. We consider the weight sequence
$$w'=(c;a_1,\dots,a_r;b_1,\dots,b_{s-1})$$
Let $d=\on{gcd}{w'}$. Let $\Omega'$ be the convex lattice domain with weight sequence
$$w(\Omega')=\left(\frac{c}{d};\frac{a_1}{d},\dots,\frac{a_r}{d};\frac{b_1}{d},\dots,\frac{b_{s-1}}{d}\right)$$
By the inductive hypothesis this is a primitive lattice polygon. We have that $w'=w(d\Omega')$ and
$$\on{gcd}\{\ell_{\Z^2}(e):e\in\on{Edge}(d\Omega')\}=d$$
The extra term $b_s$ in the weight sequence for $\Omega$ means that $\Omega$ is obtained from $d\Omega'$ by adding a single additional edge. One can see that this edge has length $b_s-dm$ for some $m\in\Z_{\geq0}$ and, as $\on{gcd}\{b_s,d\}=1$ by the assumption $\on{gcd}{w(\Omega)}=1$, we have that the lattice length of this final edge is coprime to the lattice lengths of the other edges in $d\Omega'$ and so $\Omega$ is primitive as required.
\end{proof}

\begin{prop}[{\cite[Conj.~5.7]{bwo}}] \label{prop:tc_conj} Suppose $\Omega$ is a real multiple of a lattice convex toric domain. Then $X_\Omega$ is tightly constrained if and only if $\Omega$ is a primitive lattice polygon. Equivalently, $X_\Omega$ is tightly constrained if and only if $\on{gcd}{w(\Omega)}=1$.
\end{prop}

\begin{proof} It is equivalent to show that the polarised toric surface $(Y_\Omega,A_\Omega)$ is tightly-constrained if and only if $A_\Omega$ is a primitive Cartier divisor, which is exactly the content of Lem.~\ref{lem:toric_tc} when $Y_\Omega$ is smooth. If $Y_\Omega$ is not smooth, let $\pi\colon\wt{Y}_\Omega\to Y_\Omega$ be a toric resolution of singularities. $(Y_\Omega,A_\Omega)$ is tightly-constrained if and only if the pseudo-polarised surface $(\wt{Y}_\Omega,\pi^*A_\Omega)$ is tightly-constrained by Prop.~\ref{prop:degen_toric}. The polytope $\pi^*A_\Omega$ is the same as the polytope of $A_\Omega$ -- both equal to $\Omega$ -- and so $X_\Omega$ is tightly-constrained if and only if $\Omega$ is a primitive lattice polygon.
\end{proof}

We remark that while we expressly avoided weight sequences in order to prove Prop.~\ref{prop:tc_conj} we predict that there is a good notion of weight sequences for algebraic capacities that recovers weight sequences for convex domains; c.f.~\cite{bir}.

As mentioned, the motivation for the notion of tightly-constrained domains arose in \cite{bwo} from considering explicit quasi-polynomial expressions for the cap function of $X_\Omega$. With the extra insight of Lem.~\ref{lem:toric_tc} we immediately obtain the following generalisation of \cite[Thm.~1.1]{bwo}, proved either via the methods of \cite[\S5]{bwo} or directly from Prop.~\ref{prop:cap_rep}. We say that $r\in\R$ is \emph{attained} by $(X,\omega)$ if there exists $k\in\Z_{\geq0}$ such that $\cech_k(X,\omega)=r$.

\begin{thm} \label{thm:prim_ech} Suppose $\Omega$ is a primitive lattice convex domain with $\Omega$-perimeter $\lambda$. Then there exists
some $x_0\in\Z_{\geq0}$ such that $\on{cap}_{X_\Omega}(x)$ is given by a quasipolynomial for all $x\in\Z_{\geq x_0}$. More precisely,
there exist $\gamma_0,\dots,\gamma_{\lambda-1}\in\Q$ such that
\begin{align*} \on{cap}_{X_\Omega}(i+\lambda x)&=\on{ehr}_\Omega(x)+\gamma_i \\
&=h^0(Y_\Omega,xA_\Omega)+\gamma_i
\end{align*}
when $x\in\Z_{\geq x_0}$ and $x\equiv i\on{mod}{\lambda}$. That is, for $x\in\Z_{\geq x_0}$
\begin{equation}
\on{cap}_{X_\Omega}(x)=\frac{1}{4\on{vol}(\Omega)}x^2+\frac{\ell_{\Z^2}(\partial\Omega)}{4\on{vol}(\Omega)}x+\gamma_i \tag{$\clubsuit$} \label{eqn:qp_symp}
\end{equation}
Moreover, $\on{cap}_{X_\Omega}(x)$ is given by the quasi-polynomial (\ref{eqn:qp_symp}) for all $x\geq x_0$ where $x_0$ satisfies
\begin{enumerate}
\item $x_0>2\on{vol}(\Omega)-\ell_{\Z^2}(\partial\Omega)$
\item all values $x_0,x_0+1,\dots,x_0+2\on{vol}(\Omega)-1$ are attained by $X_\Omega$. \label{item:c}
\end{enumerate}
\end{thm}

Recall that (\ref{item:c}) means that
$$\on{cap}_{X_\Omega}(x_0+i-1)<\on{cap}_{X_\Omega}(x_0+i)$$
for $i=0,\dots,2\on{vol}(\Omega)-1$. Similarly to the discussion in in \cite[Ex.~5.11]{bwo} Thm.~\ref{thm:prim_ech} also enables us to obtain quasipolynomial expressions for the cap function of any rational convex toric domain. It also computes the cap function when $\Omega$ is a real multiple of a lattice convex domain but this will not be a quasipolynomial in general; see \cite{cgk}.

\subsection{Sub-leading asymptotics for ECH}

We obtain counterparts in ECH to the previous results of this paper on sub-leading asymptotics for algebraic capacities.

To begin, specialising Lem.~\ref{lem:fin_div} to toric surfaces gives the following.

\begin{prop} \label{prop:cg} Suppose $\Omega$ is a real multiple of a lattice convex domain. Then there exists a finite list of polygons $P_0,\dots,P_n\subseteq\R^2$ such that, for sufficiently large $k$, the lattice convex domains corresponding to optimisers for $\cech_k(X_\Omega)$ are of the form $P_i+d\Omega$ for some $i$ and some $d\in\Z_{\geq0}$.
\end{prop}

This resolves the conjecture \cite[Conj.~1.4]{bwo}.

\begin{proof} It suffices to restrict to the case that $\Omega$ is a lattice convex domain. When $Y_\Omega$ is a smooth toric surface Lem.~\ref{lem:fin_div} implies that there exist divisors $D_0,\dots,D_n$ such that optimisers for $\calg_k(Y_\Omega,A_\Omega)$ are of the form $D_i+dA_\Omega$ for some $i$ and some $d\in\Z_{\geq0}$. The corresponding polygon optimisers from \cite[Thm.~1.5]{bwo} for $\cech_k(X_\Omega)$ are the polygons
$$P(D_i+dA_\Omega)=P(D_i)+d\Omega$$
and so setting $P_i=P(D_i)$ gives the result. The same continuity argument from Prop.~\ref{prop:tc_conj} gives the case when $Y_\Omega$ is singular and $\Omega$ is a lattice domain.
\end{proof}

Prop.~\ref{prop:cg} formalises the intuition that lattice paths computing $\cech_k(X_\Omega)$ should increasingly `resemble' the boundary of $\Omega$ as $k$ becomes large, which comes from viewing the combinatorial avatar of toric ECH as a type of isoperimetric problem \cite{wulff}.

An obvious obstruction to the existence of a symplectic embedding $\iota\colon(X,\omega)\to(X',\omega')$ is for
$$\on{vol}(X,\omega)>\on{vol}(X',\omega')$$
The ECH capacities asymptotically recover this constraint by the celebrated Weyl law in ECH.

\begin{thm}[{\cite[Thm.~1.1]{asy1}}] Suppose $(X,\omega)$ is a compact symplectic $4$-manifold. Then
$$\underset{k\to\infty}{\on{lim}}\frac{\cech_k(X,\omega)^2}{k}=4\on{vol}(X)$$
\end{thm}

Observe that Thm.~\ref{thm:bwo} together with Prop.~\ref{prop:sing_asy} implies the Weyl law in ECH for rational-sloped convex toric domains. A more far-reaching continuity argument in the spirit of Prop.~\ref{prop:degen_toric} would give the Weyl law in ECH via an algebraic argument for many more divisor complements using extensions of the methods in \S\ref{sec:degen} and \cite[\S2-3]{cw}. 

The Weyl law in ECH enables one to study the `sub-leading asymptotics' of ECH capacities via the error terms
$$e_k(X,\omega):=\cech_k(X,\omega)-2\sqrt{\on{vol}(X,\omega)k}$$
There has been much recent work to understand the asymptotics of $e_k(X)$, which should provide subtler numerical obstructions to the existence of symplectic embeddings; see \cite[Cor.~1.13]{hu} and Cor.~\ref{cor:ek_obs} below. We follow the convention of referring to a compact domain in $\R^4$ whose boundary is smooth and transverse to the radial vector field as a `nice star-shaped domain'. Sun in \cite{sun_ech} showed that when $(X,\omega)$ is a nice star-shaped domain
$$e_k(X,\omega)=O(k^{125/252})$$
and Cristofaro-Gardiner--Savale \cite{asy2} improved this to $e_k(X,\omega)=O(k^{2/5})$. The primary methods used in extracting these asymptotics come from Seiberg--Witten theory. For the case of general domains in $\R^4$ Hutchings \cite{hu} showed by more direct methods that $e_k(X,\omega)=O(k^{1/4})$. The author's understanding is that the expectation for all $(X,\omega)$ is
$$e_k(X,\omega)=O(1)$$
which these estimates are approaching. This is the case for all examples that have been computed.

We use Cor.~\ref{cor:lim_q} to compute the lim inf and lim sup of the error $e_k(X)$ for many non-generic convex toric domains, in particular showing that $e_k(X,\omega)$ is $O(1)$ in these cases. We denote by $\ell_{\Z^2}(v)$ the lattice length of a vector $v\in\R^2$, and by $\ell_{\Z^2}(\partial\Lambda)$ the lattice perimeter of a polygon $\Lambda$.

\begin{prop} \label{prop:lims} Suppose $\Omega=q\Omega_0$ is a real multiple of a lattice convex domain $\Omega_0$. Then,
$$\limsup_{k\to\infty}e_k(X_\Omega)=q-\frac{1}{2}\ell_{\Z^2}(\partial\Omega))\,\text{\textnormal{ and }}\liminf_{k\to\infty}e_k(X_\Omega)=-\frac{1}{2}\ell_{\Z^2}(\partial\Omega)$$
In particular, their midpoint is $\frac{q}{2}-\frac{1}{2}\ell_{\Z^2}(\partial\Omega)$.
\end{prop}

\begin{proof} It follows from Cor.~\ref{cor:lim_q} that
$$\limsup_{k\to\infty}e_k(X_\Omega)=\on{gap}(Y_\Omega,A_\Omega)+\frac{1}{2}K_{Y_\Omega}\cdot A_\Omega\text{\textnormal{ and }}\liminf_{k\to\infty}e_k(X_\Omega)=\frac{1}{2}K_{Y_\Omega}\cdot A_\Omega$$
Suppose without loss of generality that $\Omega_0$ is a primitive lattice polygon. From standard toric geometry
$$-K_{Y_\Omega}\cdot A_{\Omega_0}=\ell_{\Z^2}(\partial\Omega_0)$$
and from Prop.~\ref{prop:tc_conj} we have $\on{gap}(Y_\Omega,A_{\Omega_0})=1$, which gives
$$\limsup_{k\to\infty}e_k(X_\Omega)=q\cdot(1-\frac{1}{2}\ell_{\Z^2}(\partial\Omega_0))\text{\textnormal{ and }}\liminf_{k\to\infty}e_k(X_\Omega)=-\frac{q}{2}\ell_{\Z^2}(\partial\Omega_0)$$
and, using that both sides scale nicely with $q$, we reach
$$\limsup_{k\to\infty}e_k(X_\Omega)=q-\frac{1}{2}\ell_{\Z^2}(\partial\Omega))\text{\textnormal{ and }}\liminf_{k\to\infty}e_k(X_\Omega)=-\frac{1}{2}\ell_{\Z^2}(\partial\Omega)$$
as required.
\end{proof}

Observe that Prop.~\ref{prop:lims} generalises Hutchings' calculation in \cite[Ex.~1.2]{hu} for $B(a)$ with $q=a$ and $\ell_{\Z^2}(\partial\Omega_0)=3$. It also codifies the intuition that the sub-leading asymptotics of $\cech_k(X_\Omega)$ should contain information about the perimeter of $\Omega$. As a consequence we get the following embedding obstruction.

\begin{cor} \label{cor:ek_obs} Suppose $\Omega,\Omega'$ are convex domains that are real multiples of lattice convex domains and that have the same volume. Suppose that the open toric domain $X_\Omega^\circ$ symplectically embeds in $X_{\Omega'}$. Then $\ell_{\Z^2}(\partial\Omega)\geq \ell_{\Z^2}(\partial\Omega')$.
\end{cor}

\begin{proof} Since $X_\Omega^\circ$ symplectically embeds into $X_{\Omega'}$ we have $\cech_k(X_\Omega)\leq\cech_k(X_{\Omega'})$ for all $k$. Thus, since the volumes of $\Omega$ and $\Omega'$ are the same,
$$-\frac{1}{2}\ell_{\Z^2}(\partial\Omega)=\liminf_{k\to\infty}e_k(X_\Omega)\leq\liminf_{k\to\infty}e_k(X_{\Omega'})=-\frac{1}{2}\ell_{\Z^2}(\partial\Omega')$$
and so we must have $\ell_{\Z^2}(\partial\Omega)\geq \ell_{\Z^2}(\partial\Omega')$.
\end{proof}

Compare this to \cite[Cor.~1.13]{hu}. The adjacent remark \cite[Rmk.~1.14]{hu} shows that this is not a vacuous embedding constraint. In the same paper Hutchings considers the `Ruelle invariant' $\on{Ru}(X,\omega)$ of a nice star-shaped domain in $\R^4$. We will not recall the fairly involved definition of the Ruelle invariant here, and instead refer the reader to \cite[\S1.2]{hu}. Its relevance to sub-leading asymptotics in ECH comes from the following conjecture and theorem of Hutchings.

\begin{conjecture}[{\cite[Conjecture 1.5]{hu}}] If $(X,\omega)$ is a `generic' nice star-shaped domain in $\R^4$ then
$$\lim_{k\to\infty}e_k(X,\omega)=-\frac{1}{2}\on{Ru}(X,\omega)$$
\end{conjecture}

\begin{thm}[{\cite[Theorem 1.10]{hu}}] This conjecture is true whenever $(X,\omega)$ is a `strictly' convex or concave toric domain.
\end{thm}

A strictly convex toric domain is a convex toric domain arising from $\Omega\subseteq\R^2$ where the upper part of the boundary of $\Omega$ is the graph of a function $f$ with $f'(0)<0$ and $f''<0$. The definition of `strictly concave' is similar. As mentioned, the cases we treat in Prop.~\ref{prop:lims} are complementary to those considered by Hutchings as our convex domains are non-generic. Observe that when $Y$ is a weighted projective space of the form $\pr(1,r,s)$ and $A=\mO(d)$ the algebraic Ruelle invariant agrees with the symplectic Ruelle invariant for the corresponding convex toric domain: the ellipsoid
$$E\left(\frac{d}{s},\frac{d}{r}\right)$$
That said, it is currently unclear how to relate the algebraic and symplectic Ruelle invariants in a meaningful, geometric way.

\subsection{Bounds on error terms}

Fix a star-shaped domain $(X,\omega)$ in $\R^4$. Let $Y$ be a smooth rational surface or a possibly singular toric surface equipped with a symplectic form $\omega_A$ Poincar\'e dual to an ample $\R$-divisor $A$ on $Y$. If there exists a symplectic embedding $\iota\colon(X,\omega)\to(Y,\omega_A)$ then Thm.~\ref{thm:cw} implies
$$\cech_k(X,\omega)\leq\calg(Y,A)$$
If moreover $\on{vol}(X,\omega)=\on{vol}(Y,\omega_A)$ then
$$e_k(X,\omega)\leq\ealg_k(Y,A)$$
We have seen that $\limsup_{k\to\infty}\ealg_k(Y,A)$ is finite so long as $A$ is a real multiple of a $\Z$-divisor and hence in these cases we obtain
$$\limsup_{k\to\infty}e_k(X,\omega)<\infty$$
In particular, this applies if $(X,\omega)$ is symplectomorphic to a divisor complement
$$(Y\setminus\on{supp}(A),\omega_A|_{Y\setminus\on{supp}(A)})$$
for $(Y,A)$ as above.

\section{Minimal hypersurfaces and algebraic capacities} \label{sec:rie}

We will describe the connection of algebraic capacities to minimal (hyper)surface theory. One of the fundamental tools in sourcing and studying minimal hypersurfaces -- for example in Song's recent proof \cite{song} of Yau's conjecture -- is \textit{min-max theory}. Some of the principal objects in this theory are min-max widths, which have many striking similarities to capacities in symplectic geometry. To define these widths, we require the notion of a $p$\textit{-sweepout}. These are continuous maps
$$\Phi\colon X\to\mathcal{Z}^1$$
where $X$ is a finite-dimensional simplicial complex and $\mathcal{Z}^1$ is a certain topological space of codimension one $\Z/2$-chains on $M$, satisfying a nondegeneracy condition. $\mathcal{Z}^1$ is homotopy equivalent to $\R\pr^\infty$; denote its $\Z/2$-cohomology ring by $\Z/2[\lambda]$. With this notation, the nondegeneracy condition for $p$-sweepouts is that $(\Phi^*\lambda)^p\not=0$. We refer to \cite[\S2.3]{song} for an actual definition. One should imagine a $p$-sweepout as being a formal generalisation of a $p$-dimensional family of hypersurfaces in $M$. To a $p$-sweepout $\Phi$ one can associate its mass function $\mathbf{M}_\Phi\colon X\to\R$ given at $x$ by taking the $g$-area of the chain $\Phi(x)$.

The $p$th min-max width for a compact Riemannian manifold $(M,g)$ is then
$$\omega_p(M,g):=\inf_{\Phi}\sup\{\mathbf{M}_\Phi(x):x\in\text{dom}(\Phi)\}$$
where the infimum ranges over $p$-sweepouts $\Phi$ with `no concentration of mass' (see \cite[\S2.3]{song}), and $\on{dom}{\Phi}$ is the domain of $\Phi$. This infimum should essentially be achieved by the $g$-area of a minimal hypersurface, hence the application to problems such as Yau's conjecture.

One can make a similar construction for $p$-sweepouts of codimension two, in which one uses a space $\mathcal{Z}^2$ of codimesion two $\Z/2$-chains, which is homotopy equivalent to $\C\pr^\infty$. This produces codimension two min-max widths
$$\omega_p^2(M,h)$$
defined similarly to the min-max widths above. The central example for our context is the following. 

\begin{example} \label{ex:min_hyp} Suppose $M$ is a smooth complex projective algebraic variety equipped with an ample divisor $A$. Let $D$ be a nef (or big) divisor with $h^0(D)=p+1$. This defines a (real) codimension two $p$-sweepout for $M$ by pulling back the hyperplane sections of $M$ in the morphism to $\pr^p$ by the linear system $|D|$.
\end{example}

In particular, it follows that the codimension two min-max weights in this situation satisfy
$$\omega_p^2(M,g)\leq\calg_p(M,A)$$
where $g$ is the metric corresponding to $A$. It is conceivable that this is actually an equality, and that further ties are present between the theory of minimal hypersurfaces and algebraic capacities. As was the case for ECH capacities, algebraic capacities are often more computable -- when the nef cone is well-behaved -- or at least provide readily available estimates.

\end{document}